\documentclass[11pt]{article}
\usepackage{graphicx}
\usepackage{times}
\usepackage{amsmath,amsfonts,amssymb,amsthm}
\title{Random conical tessellations}
\author{Daniel Hug\footnote{Supported in part by DFG grants FOR 1548 and HU 1874/4-2.}\;  and Rolf Schneider }
\date{}
\newcommand{\Sd}{{\mathbb S}^{d-1}}
\newcommand{\Rd}{{\mathbb R}^d}
\newcommand{\D}{{\rm d}}
\newcommand{\fed}{\,\rule{.1mm}{.26cm}\rule{.24cm}{.1mm}\,}
\newcommand{\R}{{\mathbb R}}
\newcommand{\eps}{\varepsilon}
\newcommand{\Hc}{\mathcal H}

\newcommand{\E}{{\mathbb E}\,}

\newtheorem{theorem}{Theorem}[section]
\newtheorem{lemma}{Lemma}[section]
\newtheorem{proposition}{Proposition}[section]
\newtheorem{corollary}{Corollary}[section]

\newtheorem{definition}{Definition}[section]

\begin{document}
\maketitle

\begin{abstract}
We consider tessellations of the Euclidean $(d-1)$-sphere by $(d-2)$-dimensional great subspheres or, equivalently, tessellations of Euclidean $d$-space by hyperplanes through the origin; these we call conical tessellations. For random polyhedral cones defined as typical cones in a conical tessellation by random hyperplanes, and for random cones which are dual to these in distribution, we study expectations for a general class of geometric functionals. They include combinatorial quantities, such as face numbers, as well as, for example, conical intrinsic volumes. For isotropic conical tessellations (those generated by random hyperplanes with spherically symmetric distribution), we determine the complete covariance structure of the random vector whose components are the $k$-face contents of the induced spherical random polytopes. This result can be considered as a spherical counterpart of a classical result due to Roger Miles. 

\medskip


\medskip

\end{abstract}

\section{Introduction}\label{sec1}

A major theme of stochastic geometry, since the seminal work of R\'{e}ny and Sulanke in 1963/64, has always been the
investigation of geometric functionals of random convex polytopes. The survey articles \cite{Rei10}, \cite{Hug13}, \cite{Sch16b} give an impressive picture of the progress in recent years. They also reveal that, as far as expectations and higher moments, a prerequisite for the study of limit theorems, are concerned, one generally has to be satisfied with asymptotic results and estimates, whereas explicit results are very rare. 

Most of the random polytopes studied so far live in Euclidean spaces. In other spaces of constant curvature, several results may have parallel versions, but also new phenomena are to be expected, in particular in spherical space due to its compactness. A recent study \cite{BHRS15} of spherically convex hulls of random points in $\Sd$ already exhibited some phenomena which cannot be observed in Euclidean spaces. The present paper is devoted to random polytopes in the unit sphere $\Sd$ of Euclidean space $\Rd$. For basic classes of random convex polytopes in $\Sd$, we find explicit formulas for the first and mixed second moments of a series of quite general geometric functionals. The spherically convex polytopes in $\Sd$ are in one-to-one correspondence with their positive hulls, which are convex polyhedral cones in $\Rd$. Thus, the study of random polytopes in the sphere is equivalent to the study of random polyhedral convex cones in Euclidean space. The geometry of polyhedral cones has recently found increased interest, due to applications in convex optimization and compressed sensing (see, e.g.,  \cite{AB15}, \cite{AL14}, \cite{ALMT14}, \cite{DH10}, \cite{GNP14}, \cite{MT14}). 

Let us first describe the random polytopes in $\Sd$ and the geometric functionals of them that we consider. First, take $n\ge d$ independent, identically distributed random points in $\Sd$. Their distribution need only satisfy some mild requirements, besides evenness they guarantee general position with probability one. The spherically convex hull of the random points, under the condition that it is not the whole sphere, defines a random polytope. It was first studied by Cover and Efron \cite{CE67}. Therefore, we call its positive hull a Cover--Efron cone. In distribution, this random cone is dual to the random Schl\"afli cone, which we define as follows. To the given random vectors in the unit sphere, we consider the orthogonal hyperplanes through the origin. They induce a random tessellation of $\Rd$ into convex cones. Among its $d$-dimensional cones, we choose one at random, with equal chances. This defines what we call a random Schl\"afli cone. Its intersection with $\Sd$ yields the second type of spherical random polytope that we consider, again following Cover and Efron. 

For a spherical polytope $P$, contained in an open hemisphere, the $j$th quermassintegral $U_j(P)$ is, up to a normalizing factor, the total invariant measure of the set of $(n-j)$-flats through the origin that meet $P$. Then, we define $Y_{k,j}(P)$ as the sum of $U_j(F)$ over all $(k-1)$-faces $F$ of $P$ (or correspondingly for polyhedral cones). These general functionals comprise combinatorial functionals, such as numbers of $k$-faces, as well as metric functionals, such as total $k$-face contents, and they allow to express the $k$th conical intrinsic volume. These conical, or spherical, intrinsic volumes appeared first, with different terminology, in Santal\'{o}'s work on integral geometry and the Gauss--Bonnet formula in spherical spaces, for example, in \cite{San62a}, \cite{San62}. To the linear relations between the spherical intrinsic volumes listed in \cite{San62a},  McMullen \cite{McM75} later found, in the case of polyhedral cones, a new combinatorial approach. For later appearances of the spherical intrinsic volumes in spherical geometry, we refer to \cite{Gla95}, \cite{Gla96}, \cite[Section 6.5]{SW08}, \cite{GHS03}. More recently, the conical intrinsic volumes, and also their integral geometry, have found very interesting applications in convex optimization and compressed sensing. We refer to \cite{AB15}, \cite{ALMT14}, \cite{GNP14}, \cite{MT14}. As a sequel to this, new approaches to, and new perspectives on, conical intrinsic volumes of polyhedral cones came forward, with relations to combinatorial aspects being in the foreground; see \cite{Ame15}, \cite{AL15}. We emphasize, however, that the following is meant as a contribution to stochastic geometry, where first and higher moments of geometric random variables are in the focus of interest, often as a first step towards more sophisticated distribution and limit results.

In the following, after introducing the announced random cones and geometric functionals and some of their properties, we first extend the work of Cover and Efron by determining the expectations of the functionals $Y_{k,j}$ for random Schl\"afli cones. By specialization, this yields the results of Cover and Efron on face numbers, and also new results, such as for the conical intrinsic volumes. By dualization, corresponding results for the Cover--Efron cones are obtained. The major part of this paper is devoted to the functionals $\Lambda_k= Y_{k,k-1}$ of a polyhedral cone. For a spherical polytope $P\subset\Sd$, the value $\Lambda_{k+1}({\rm pos}\,P)$ is the {\em total $k$-face content}, that is, the sum of the $k$-dimensional normalized Lebesgue measures of the $k$-faces of $P$, in other words, the $k$-dimensional normalized Hausdorff measure of its $k$-skeleton. As examples, for $k=0,1,d-2,d-1$ we get, respectively, the vertex number and, up to constant factors, the total edge length, the  surface area and the volume of $P$. Thus, these functionals interpolate, in a natural way, between vertex number and  volume. Recently, Amelunxen (\cite{Ame15}, with different notation) has proved kinematic formulas for these functionals in the case of polyhedral cones. The expectations of the $\Lambda_k$ for a random Schl\"afli cone are special cases of our results mentioned above. 

Our main result is the determination of the complete covariance structure of the sequence $\Lambda_0(S),\dots,\Lambda_d(S)$ for an isotropic random Schl\"afli cone $S$. This is a conical counterpart to a result of Miles, who in \cite{Mil61} considered the typical cell of a stationary, isotropic Poisson hyperplane mosaic in $\Rd$ and determined all mixed moments of its total face contents. Miles presented his result also in \cite[formula (63)]{Mil70}. As remarked in \cite{Sch16}, the proof given by Miles in \cite{Mil61} makes heavy use of ergodic theory and is not explicitly carried out in all details. A simpler proof was given in \cite{Sch16}, where the result of Miles was extended to the non-isotropic case and to typical faces of lower dimensions. Our proof in the following carries over an idea of Miles to the conical case, but is essentially different in the details.
 
Since our random Schl\"afli cones are induced by random hyperplanes through the origin, this paper is also a contribution to random conical tessellations (which explains the title), or equivalently to tessellations of the sphere by random great subspheres, yielding special spherical mosaics. Random mosaics in Euclidean spaces are an intensively studied topic of stochastic geometry. We refer the reader to Chapter 10 in the book \cite{SW08} and to the more recent survey articles  \cite{Cal13}, \cite{Hug13}, \cite{VGS13} and \cite{RL15}. A much investigated particular class, besides the Voronoi tessellations, are hyperplane tessellations, in particular those generated by stationary Poisson processes of hyperplanes, initiated by the seminal work of Miles \cite{Mil61}, \cite{Mil64a}, \cite{Mil64b}, \cite{Mil70} and Matheron \cite{Mat74}, \cite{Mat75}. Relatively little has been done on random tessellations of spaces other than the Euclidean. Tessellations of the sphere of arbitrary dimension by great subspheres (of codimension $1$) were briefly considered by Cover and Efron \cite{CE67}, and those of the two-dimensional sphere in more detail by Miles \cite{Mil71}; see in particular Theorem 6.3 on some mixed second moments, which is widely generalized by our result. Relations between various densities of random mosaics in spherical spaces were studied by Arbeiter and Z\"ahle \cite{AZ94}.

In Section \ref{sec2} we introduce the geometric functionals of polyhedral cones that will be studied, and in Section \ref{sec3} the two types of random cones for which we investigate first and second moments of these functionals. Expectation results for the functionals $Y_{k,j}$, which extend formulas of Cover and Efron, are derived in Section \ref{sec4}. Sections \ref{sec5} to \ref{sec7} are then preparatory to our main result on mixed second moments, which is finally obtained in Section \ref{sec8}. Hints to the proof strategy are given at the beginning of Sections \ref{sec6} and \ref{sec8}.

\section{Geometric functionals of convex cones}\label{sec2}

We work in $d$-dimensional Euclidean space $\Rd$ ($d\ge 2$), with scalar product $\langle\cdot\,,\cdot \rangle$, and denote by $\Sd$ its unit sphere. Let $\sigma_m$, $m\in\mathbb{N}_0$, be the $m$-dimensional spherical Lebesgue measure (i.e., the $m$-dimensional Hausdorff measure) on $m$-dimensional great subspheres of $\Sd$. For 
$n\in\mathbb{N}$ we put 
$$ \omega_n := \sigma_{n-1}(\mathbb{S}^{n-1}) = \frac{2\pi^{n/2}}{\Gamma(n/2)}.$$ 

Let ${\mathcal C}^d$  denote the set of (nonempty) closed convex cones in $\Rd$, which includes $k$-dimensional linear subspaces, $k\in\{0,\ldots,d\}$. We equip ${\mathcal C}^d$ with the topology induced by the Fell topology (see \cite[Sec.~12.2]{SW08}), or equivalently, with the topology induced by the Euclidean Hausdorff distance restricted to the intersections of the cones in ${\mathcal C}^d$ with the unit ball centered at the origin. A cone $C\in{\mathcal C}^d$ is called {\em pointed} if it does not contain a line. We write ${\mathcal {PC}}^d$ for the set of polyhedral cones in ${\mathcal C}^d$. This set is a Borel subset of ${\mathcal C}^d$. For $C\in{\mathcal {PC}}^d$ and for $k\in\{0,\dots,d\}$, we denote by ${\mathcal F}_k(C)$ the set of $k$-dimensional faces of $C$.

For $C\in{\mathcal C}^d$, the dual cone is defined by
$$ C^\circ:= \{y\in \Rd: \langle y,x\rangle \le 0 \mbox{ for all }x\in C\}.$$
This is again a cone in ${\mathcal C}^d$, and $C^{\circ\circ}:= (C^\circ)^\circ =C$. If $C$ is pointed and $d$-dimensional, then $C^\circ$ has the same properties. If $C\in{\mathcal {PC}}^d$ and $F\in{\mathcal F}_k(C)$ for $k\in\{0,\dots,d\}$, then the normal cone $N(C,F)$ of $C$ at $F$ is a $(d-k)$-face of the polyhedral cone $C^\circ$, also called the {\em conjugate face} (of $F$ with respect to $C$) and denoted by $\widehat F_C$. If $\widehat F_C=G$, then $\widehat G_{C^\circ}=F$.

The following fact is occasionally useful. We give a proof for convenience.

\begin{lemma}\label{L2.1}
Suppose that $C\in{\mathcal C}^d$ is pointed, and let $L\subset \Rd$ be a linear subspace. Then
$$ L\cap C\not=\{0\} \;\Leftrightarrow\, L^\perp\cap{\rm int}\,C^\circ=\emptyset.$$
\end{lemma}

\begin{proof}
Suppose that $L\cap C\not=\{0\}$. Choose $v\in L\cap C$, $v\not= 0$. Suppose there exists $y\in L^\perp\cap{\rm int}\,C^\circ$. Since $y\in L^\perp$, we have $\langle y,v\rangle =0$. Since $y\in {\rm int}\,C^\circ$, the points $y'$ in some neighbourhood of $y$ belong to $C^\circ$ and hence satisfy $\langle y',v\rangle \le 0$. But since $\langle y,v\rangle =0$ and $v\not=0$, this is impossible.

Suppose that $L^\perp\cap{\rm int}\,C^\circ=\emptyset$. The disjoint convex sets $L^\perp$ and ${\rm int}\,C^\circ$ can be separated by a hyperplane, hence there is a vector $v\not=0$ with $\langle v, y\rangle \le 0$ for all $y\in {\rm int}\,C^\circ$ and $\langle v,z\rangle \ge 0$ for all $z\in L^\perp$; the latter implies $\langle v,z\rangle = 0$ for $z\in L^\perp$ and thus $v\in L$. Since $C$ does not contain a line, ${\rm int}\,C^\circ\not=\emptyset$, hence $\langle v, y\rangle \le 0$ holds for all $y\in C^\circ$. Therefore, $v\in C^{\circ\circ}=C$. Thus, $v\in L\cap C$.
\end{proof}

A set $M\subset\Sd$ is {\em spherically convex} if ${\rm pos}\,M$ is convex; here ${\rm pos}$ denotes the positive hull.
To include some degenerate cases in the following, we define ${\rm pos}\,\emptyset:=\{0\}$. If $C\in{\mathcal C}^d$, the set $K=C\cap \Sd$ is called a {\em convex body} in $\Sd$, and we have $C={\rm pos}\, K$. In particular, the empty set and $k$-dimensional great subspheres, that is, intersections of $(k+1)$-dimensional linear subspaces with $\Sd$, for $k\in\{0,\ldots,d-1\}$ (and thus including ${\mathbb S}^{d-1}$), are convex bodies in $\Sd$. The set of convex bodies in $\Sd$ is denoted by ${\mathcal K}_s$ (this notation, as well as the term `convex body', differs from the usage in  \cite[Sec.~6.5]{SW08}, where the empty set is excluded). For $K\in {\mathcal K}_s$, the dual convex body $K^\circ$ is defined by
$$ K^\circ:= \{y\in \Sd: \langle y,x\rangle \le 0 \mbox{ for all }x\in K\} = ({\rm pos}\, K)^\circ \cap \Sd.$$ 

To introduce the conical quermassintegrals and the conical intrinsic volumes, we make use of the correspondence between convex cones in $\Rd$ and spherically convex sets in $\Sd$. For the latter, the functionals to be considered were already introduced by Santal\'{o}, see \cite[Part IV]{San76}, with different notation. We follow here the approach of Glasauer \cite{Gla95} and refer to \cite[Sec.~6.5]{SW08} for further details.  

Let $G(d,k)$ denote the Grassmannian of $k$-dimensional linear subspaces of $\Rd$, and let $\nu_k$ be its normalized Haar measure (the unique rotation invariant Borel probability measure on $G(d,k)$), $k=0,\dots,d$. For $K\in {\mathcal K}_s$, the {\em spherical quermassintegrals}  are defined by
\begin{equation}\label{01} 
U_j(K):=\frac{1}{2} \int_{G(d, d-j)} \chi(K\cap L)\,\nu_{d-j}(\D L), \quad j=0,\dots,d, 
\end{equation}
where $\chi$ denotes the Euler characteristic. (Of course, $U_0(K)=\frac{1}{2}\chi(K)$ and $U_d(K)=0$, but this is included for formal reasons). These are, essentially, the `Grassmann angles' of Gr\"unbaum \cite{Gru68}, who derived for them various polyhedral relations. We recall from \cite[p.~262]{SW08} that if $K$ is a convex body in $\Sd$ and not a great subsphere, then $\chi(K\cap L)=\mathbf{1}\{K\cap L\neq\emptyset\}$ for $\nu_{d-j}$ almost all $L\in G(d,d-j)$. Hence, in this case $2\,U(K)$ is the total invariant probability measure of the set of all $(d-j)$-dimensional linear subspaces hitting $K$. Since $\chi({\mathbb S}^k)=1+(-1)^k$ for a great subsphere ${\mathbb S}^k$ of dimension $k\in \{0, \dots, d-1\}$, we have
$$ U_j({\mathbb S}^k) =\left\{ \begin{array}{ll} 1, & \mbox{if } k-j\ge 0 \mbox{ and even},\\[2mm]
0, & \mbox{if }k-j<0 \mbox{ or odd}.\end{array}\right.
$$ 

For cones $C\in {\mathcal C}^d$, we now define
\begin{equation}\label{02}  
U_j(C):= U_j(C\cap\Sd).
\end{equation}
If $C\in  {\mathcal C}^d$ is not a linear subspace, then
\begin{equation}\label{sub1}
U_j(C) =  \frac{1}{2} \int_{G(d, d-j)} \mathbf{1}\{ C  \cap L\neq\{0\}\}\,\nu_{d-j}(\D L), \quad j=0,\dots,d.
\end{equation}
If $L^k\subset{\mathbb R}^d$ is a linear subspace of dimension $k$, then
\begin{equation}\label{03}
U_j(L^k) =\left\{ \begin{array}{ll}1, & \mbox{if } k-j>0 \mbox{ and odd},\\[2mm]
0, & \mbox{if }k-j\le 0\mbox{ or even}.\end{array}\right.
\end{equation}

Let $0\le j\le m\le d-1$, let $M\subset\Rd$ be an $m$-dimensional linear subspace and $C\in{\mathcal C}^d$ a cone with $C\subset M$. The image measure of $\nu_{d-j}$ under the map $L\mapsto L\cap M$ from $G(d,d-j)$ to the Grassmannian of $(m-j)$-subspaces in $M$ is the normalized Haar measure on the latter space.  Here (and subsequently) we tacitly use the fact that $\nu_{d-j}(\{L\in G(d,d-j):L\cap M\notin G(d,m-j)\})=0$; see \cite[Lemma 13.2.1]{SW08}. Therefore, it follows from (\ref{01}), (\ref{02}) that $U_j(C)$ does not depend on whether it is computed in $\Rd$ or in $M$.

In particular, for $C\in{\mathcal C}^d$ and $m\in\{1,\dots,d\}$, we have
$$ \dim C\le m\;\Rightarrow\; U_{m-1}(C)=\frac{\sigma_{m-1}(C\cap\Sd)}{\omega_m}.$$

If $C\in {\mathcal C}^d$ is not a linear subspace, the duality relation
\begin{equation}\label{2.4}
U_j(C)+U_{d-j}(C^\circ) = \frac{1}{2}
\end{equation}
holds for $j=0,\dots,d$. If $C$ is pointed and $d$-dimensional, this follows from (\ref{sub1}) and Lemma \ref{L2.1}. If $C\in{\mathcal C}^d$ is not a subspace, the assertion can be obtained from the previous case by approximation, using easily established continuity properties. If $C$ is a subspace, duality is of little interest, in view of (\ref{03}).

We now recall the spherical intrinsic volumes and refer to \cite[ Sec.~6.5]{SW08} for details. Let $d_s$ be the spherical distance on $\Sd$; thus, for $x,y\in \Sd$, $d_s(x,y)=\arccos\,\langle x,y\rangle$. For $K\in{\mathcal K}_s\setminus\{\emptyset\}$ and $x\in \Sd$, the distance of $x$ from $K$ is  $d_s(K,x):=\min\{d_s(y,x):y\in K\}$.  For $0<\varepsilon<\pi/2$, the (outer) parallel set of $K$ at distance $\varepsilon$ is defined by 
$$ M_\varepsilon(K) := \{x\in\Sd: 0<d_s(K,x)\le\varepsilon\}.$$
By the spherical Steiner formula, the measure of this set can be written in the form
$$ \sigma_{d-1}(M_\varepsilon(K)) =\sum_{m=0}^{d-2} g_{d,m}(\varepsilon) v_m(K)$$
with 
$$ g_{d,m}(\varepsilon) := \omega_{m+1}\omega_{d-m-1}\int_0^\varepsilon \cos^m\varphi\sin^{d-m-2}\varphi\,\D \varphi$$
for $0\le\varepsilon<\pi/2$. This defines the numbers $v_0(K),\dots,v_{d-2}(K)$ uniquely.  The definition is supplemented by setting $v_m(\emptyset):=0$,
$$v_{d-1}(K):=\frac{\sigma_{d-1}(K)}{\omega_d},$$ 
and 
$$ v_{-1}(K) := v_{d-1}(K^\circ).$$
Note that $v_m(\Sd)=0$ for $m=0,\ldots,d-2$ and $v_{d-1}(\Sd)=1$. 
The numbers $v_i(K)$ are the {\em spherical intrinsic volumes} of $K$. In particular, for $K\in{\mathcal K}_s$ and $m=0,\ldots,d-1$,
$$  \dim K\le m\;\Rightarrow\; v_m(K) = \frac{\sigma_m(K)}{\omega_{m+1}}.$$

For spherical polytopes, the spherical intrinsic volumes have representations in terms of angles, similar as in the Euclidean case. For a spherical polytope $P$ and for $k\in\{0,\dots,d-2\}$, we denote by ${\mathcal F}_k(P)$ the set of $k$-faces of $P$. Let $P$ be a spherical polytope and $F\in{\mathcal F}_k(P)$. The {\em external angle} $\gamma(F,P)$ of $P$ at $F$ is defined by
$$ \gamma(F,P):= \gamma({\rm pos}\,F ,{\rm pos}\,P) := \frac{\sigma_{d-k-2}(N({\rm pos}\,P,{\rm pos}\,F)\cap \Sd)}{\omega_{d-k-1}}.$$
With these notations, we have
$$ v_m(P)= \frac{1}{\omega_{m+1}} \sum_{F\in{\mathcal F}_m(P)} \sigma_m(F)\gamma(F,P),\qquad m=0,\ldots,d-2.$$

For cones $C \in {\mathcal C}^d$, the {\em conical intrinsic volumes} are now defined by
$$ V_m(C):= v_{m-1}(C\cap \Sd),\qquad m=0,\dots,d.$$
The shift in the index has the advantage that the highest occurring index is equal to the maximal possible dimension of $C$. Since $C$ is a cone, there is no danger of confusion with the intrinsic volumes of compact convex bodies. 

For a cone $C\in{\mathcal C}^d$ with $\dim C=k$, the {\em internal angle} of $C$ at $0$ is defined by
$$ \beta(0,C) =\frac{\sigma_{k-1}(C\cap \Sd)}{\omega_k}.$$
Then, for an arbitrary polyhedral cone $C\in\mathcal {PC}^d$ and for $m=1,\dots,d-1$, we have
$$
V_m(C) = \sum_{F\in{\mathcal F}_m(C)} \beta(0,F)\gamma(F,C).
$$
In particular, if $ \dim C= m$, then $V_m(C)= \beta(0,C)$. 

In contrast to the quermassintegrals and intrinsic volumes of convex bodies in Euclidean space, which differ only by their normalizations, the conical quermassintegrals and conical intrinsic volumes are essentially different functionals. However, they are closely related. A spherical integral-geometric formula of Crofton type (see \cite[(6.63)]{SW08}) implies that
\begin{equation}\label{2.2} 
U_j(C)= \sum_{k=0}^{\lfloor \frac{d-1-j}{2}\rfloor} V_{j+2k+1}(C)
\end{equation}
for $C\in {\mathcal C}^d$ and $j=0,\dots,d-1$. From \eqref{2.2} it follows that
\begin{equation}\label{2.3} 
\left.
\begin{array}{lll}
V_j  &=& U_{j-1}-U_{j+1} \quad\mbox{ for } j=1,\dots,d-2,\\
V_{d-1}&=& U_{d-2},\\
V_d&=& U_{d-1}.
\end{array}\right\}
\end{equation}

The duality relation 
\begin{equation}\label{2.1}
V_m(C)=V_{d-m}(C^\circ),\quad m=0,\dots,d
\end{equation}
holds for $C\in{\mathcal C}^d$. For $m\in\{0,d\}$ it holds by definition. For $m\in\{1,\dots,d-1\}$, it follows from (\ref{2.4}) and (\ref{2.3}) if $C$ is not a subspace, and from
\begin{equation}\label{2.1c} 
V_j(L^k) =\delta_{jk}
\end{equation}
(Kronecker symbol) if $C=L^k$ is a $k$-dimensional subspace; here (\ref{2.1c}) follows from (\ref{03}) and (\ref{2.3}).

As did Miles \cite[Sec.~5.8]{Mil61} for convex polytopes in $\Rd$, we use the conical quermassintegrals to define a more general series of functionals for polyhedral cones, which comprises the geometrically most interesting functionals as special cases. For $C\in{\mathcal PC}^d$, $k=1,\dots,d$ and $j=0,\dots,k-1$, let
\begin{equation}\label{2.5} 
Y_{k,j}(C):= \sum_{F\in{\mathcal F}_k(C)} U_j(F).
\end{equation}
Then, in particular,
$$
Y_{\dim C,j}(C)= U_j(C).
$$
According to (\ref{2.3}), also the conical intrinsic volumes can be expressed in terms of suitable functions $Y_{k,j}$.

If $C\in {\mathcal PC}^d$ is such that the $k$-faces of $C$ are not linear subspaces, then 
\begin{equation}\label{2.5b} 
Y_{k,0}(C) =\frac{1}{2}f_k(C),
\end{equation}
where $f_k(C)$ denotes the number of $k$-faces of $C$. 

We see that for a $d$-dimensional pointed polyhedral cone both, the combinatorial functionals given by the face numbers and the metric functionals given by the conical intrinsic volumes, can be expressed in terms of suitable functionals $Y_{k,j}$.

Further, for  $C\in{\mathcal PC}^d$ and $k\in\{1,\dots,d\}$, we define the functional $\Lambda_k$ by 
\begin{equation}\label{2.6}
\Lambda_k(C) :=  \sum_{F\in{\mathcal F}_k(C)} V_k(F).
\end{equation} 
As explained in the introduction, $\Lambda_k$ can be considered as the total $k$-face content, also for a polyhedral cone, if `content' is interpreted properly.
Since the conical intrinsic volumes and  the conical quermassintegrals are intrinsically defined, it follows from \eqref{2.3} that
$$
Y_{k,k-1}(C) = \Lambda_k(C).
$$

\section{Conical tessellations and the Cover--Efron model}\label{sec3}

In this section, we introduce random conical tessellations and the two basic types of random polyhedral cones that they induce. These random cones were first considered by Cover and Efron \cite{CE67}. We slightly modify and formalize the approach of \cite{CE67}, to meet our later requirements.

Recall that $G(d,d-1)$ denotes the Grassmannian of $(d-1)$-dimensional linear subspaces of $\Rd$. We say that hyperplanes $H_1,\dots,H_n\in G(d,d-1)$ are {\em in general position} if any $k\le d$ of them have an intersection of dimension $d-k$. For a vector $x\in \Rd\setminus\{0\}$, let
$$
x^\perp = \{y\in\Rd:\langle y,x\rangle =0\},\qquad 
x^- = \{y\in\Rd:\langle y,x\rangle \le 0\}.
$$

We shall repeatedly make use of the duality
\begin{equation}\label{3.1} 
({\rm pos}\{x_1,\dots,x_n\})^\circ=\bigcap_{i=1}^n x_i^-,\qquad {\rm pos}\{x_1,\dots,x_n\}=\left(\bigcap_{i=1}^n x_i^-\right)^\circ
\end{equation}
for $x_1,\dots,x_n\in\Rd$.

Vectors $x_1,\dots,x_n\in\Rd$ are said to be {\em in general position} if any $d$ or fewer of these vectors are linearly independent. Thus, the hyperplanes $x_1^\perp,\dots,x_n^\perp$ are in general position if and only if $x_1,\dots,x_n$ are in general position. If this is the case, then
\begin{equation}\label{3.00}
{\rm pos}\{x_1,\dots,x_n\} \not=\Rd \;\Leftrightarrow\; \bigcap_{i=1}^n x_i^- \not=\{0\}\; \Leftrightarrow \;\dim \bigcap_{i=1}^n x_i^-=d,
\end{equation}
where the last implication $\Rightarrow$  follows from general position. In fact, suppose that $C:=\bigcap_{i=1}^n x^-_i$ satisfies $0<k=\dim C<d$. Let $L_k={\rm lin}\,C$. 
Choose $p\in \text{relint}\, C$ and define $I:=\{i\in \{1,\ldots,n\}:p\in x_i^\perp\}$, hence $p\in \text{int}\, x_j^-$ for $j\in\{1,\ldots,n\}\setminus I$. Then 
$C\subset \bigcap_{i\in I}x_i^\perp$ implies that $L_k\subset \bigcap_{i\in I}x_i^\perp\subset \bigcap_{i\in I}x_i^-$. Since $p\in \text{int}\, x_j^-$ for $j\in\{1,\ldots,n\}\setminus I$, 
we also have $\bigcap_{i\in I}x_i^-\subset L_k$, and thus $L_k=\bigcap_{i\in I}x_i^\perp=\bigcap_{i\in I}x_i^-$ and $L_k^\perp=\text{pos}\{x_i:i\in I\}$, by \eqref{3.1}. 
But then necessarily $|I|\ge d-k$. The assumption of general position implies that $|I|=d-k$, which is a contradiction to $L_k^\perp=\text{pos}\{x_i:i\in I\}$.

Suppose that $H_1,\dots,H_n\in G(d,d-1)$ are in general position. Then the hyperplanes  $H_1,\dots,H_n$ induce a tessellation ${\mathcal T}$ of $\Rd$ into $d$-dimensional polyhedral cones.   We call ${\mathcal T}$ a {\em conical tessellation} of $\Rd$. For $k\in\{1,\ldots,d\}$, the set of $k$-faces of $\mathcal{T}$ is defined as the union of the sets of $k$-faces  
of these polyhedral cones (the $d$-dimensional cones are the $d$-faces). We write ${\mathcal F}_k(H_1,\dots,H_n)$ for the set of $k$-faces of the tessellation ${\mathcal T}$.  Later, we shall often abbreviate $(H_1,\dots,H_n)=:\eta_n$ and then write $\mathcal F_k(\eta_n)$ for ${\mathcal F}_k(H_1,\dots,H_n)$. By $f_k({\mathcal T})$ we denote  the number of $k$-faces of the tessellation ${\mathcal T}$. 

The spherical polytopes $C\cap \Sd$, where $C$ is a cone of ${\mathcal T}$, form a tessellation of the sphere $\Sd$, or spherical tessellation. In the following, it will be more convenient to work with convex cones than with their intersections with $\Sd$.

If we denote by $H^-$ one of the two closed halfspaces bounded by the hyperplane $H$, then it follows from (\ref{3.00}) that the $d$-dimensional cones of the tessellation ${\mathcal T}$ induced by $H_1,\dots,H_n$ are precisely the cones different from $\{0\}$ of the form 
$$ \bigcap_{i=1}^n\eps_i H_i^-,\quad \eps_i=\pm 1.$$
We call these cones the {\em Schl\"afli cones} induced by $H_1,\dots,H_n$, $n\ge 1$, because Schl\"afli (generalizing a result of Steiner) has shown that there are exactly
\begin{equation}\label{3.0} 
C(n,d):= 2\sum_{r=0}^{d-1} \binom{n-1}{r}
\end{equation}
of them (the simple inductive proof is reproduced in \cite[Lem.~8.2.1]{SW08}; also references are found there). 
We consistently define $C(0,d):=1$ (where the only cone is $\R^d$ itself) and $C(n,d):=0$ for $n<0$. 

Each choice of $d-k$ indices $1\le i_1<\dots< i_{d-k}\le n$ determines a $k$-dimensional subspace $L= H_{i_1}\cap\dots\cap H_{i_{d-k}}$. For $i\in\{1,\dots,n\}\setminus \{i_1,\dots,i_{d-k}\}$, the intersections of $L$ with the hyperplanes $H_i$ are in general position in $L$ and hence determine $C(n-d+k,k)$ Schl\"afli cones with respect to $L$. Each of these is a $k$-face of the tessellation ${\mathcal T}$, and each $k$-face of ${\mathcal T}$ is obtained in this way. Thus, the total number of $k$-faces is given by 
\begin{equation}\label{3.1a}
f_k({\mathcal T}) = \binom{n}{d-k}C(n-d+k,k)=: C(n,d,k),
\end{equation}
for $k=1,\dots,d$. In particular, $f_k(\mathcal{T})=1$ if $n=d-k$ and  $f_k(\mathcal{T})=0$ if $n<d-k$. 

Now we turn to random cones. The random vectors appearing in the following can be assumed as unit vectors, since only their spanned rays are relevant. All measures on $\Sd$ or $G(d,d-1)$ appearing in the following are Borel measures. Generally, we denote by ${\mathcal B}(T)$ the $\sigma$-algebra of Borel sets of a given topological space $T$. Let $\phi$ be a probability measure on $\Sd$ which is symmetric with respect to $0$ (also called {\em even}) and assigns measure zero to each $(d-2)$-dimensional great subsphere. Let $X_1,\dots,X_n$ be independent random points in $\Sd$ with distribution $\phi$. With probability $1$, they are in general position. In the following, we denote probabilities by ${\mathbb P}$ and expectations by ${\mathbb E}$.

From Schl\"afli's result (\ref{3.0}), Wendel has deduced that 
\begin{equation}\label{3.2} 
p_n^{(d)}:= {\mathbb P}({\rm pos}\{X_1,\dots,X_n\}\not=\Rd) =\frac{C(n,d)}{2^n}
\end{equation}
(see \cite[Thm.~8.2.1]{SW08}). This result, having an essentially geometric core, does not depend on the choice of the distribution $\phi$, as long as the latter has the specified properties.  

Cover and Efron \cite{CE67} have considered the spherically convex hull of $X_1,\dots,X_n$, under the condition that this convex hull is different from the whole sphere. We talk of the {\em Cover--Efron model} if a spherically convex random polytope or its spanned cone is generated in this way. 

\begin{definition}\label{D3.1}
Let $\phi$ be as above. Let $n\in{\mathbb N}$ and let $X_1,\dots,X_n$ be independent random points with distribution $\phi$. The 
$$ (\phi,n)\mbox{-Cover--Efron cone } C_n$$
is the random cone defined as the positive hull of $X_1,\dots,X_n$ under the condition that this is different from $\Rd$.
\end{definition}

Thus, $C_n$ is a random convex cone with distribution given by ${\mathbb P}(C_n={\mathbb R}^d)=0$ and
\begin{equation}\label{3.2a} 
{\mathbb P}(C_n\in B)=\frac{1}{p_n^{(d)}} \int_{(\Sd)^n} {\bf 1}_B({\rm pos}\{x_1,\dots,x_n\})\, \phi^n(\D(x_1,\dots,x_n))
\end{equation}
for $B\in{\mathcal B}(\mathcal {PC}^d_p)$, where $\mathcal {PC}^d_p:=\mathcal {PC}^d\setminus\{\R^d\}$. Hence, $C\in B \subset \mathcal {PC}^d_p$ implies $C\not=\Rd$. 

By duality, the Cover--Efron model is connected to random conical tessellations, as we now explain.

Let $\phi^*$ be the image measure of $\phi$ under the mapping $x\mapsto x^\perp$ from the sphere $\Sd$ to the Grassmannian $G(d,d-1)$. Every probability measure $\phi^*$ on $G(d,d-1)$ that assigns measure zero to each set of hyperplanes in $G(d,d-1)$ containing a fixed line is obtained in this way. Let $\Hc_1,\dots,\Hc_n$ be independent random hyperplanes in $G(d,d-1)$ with distribution $\phi^*$. With probability $1$, they are in general position. 

\begin{definition}\label{D3.2}
Let $\phi^*$ be as above. Let $n\in{\mathbb N}$ and let $\Hc_1,\dots,\Hc_n$ be independent random hyperplanes with distribution $\phi^*$. The  
$$ (\phi^*,n)\mbox{-Schl\"afli cone } S_n$$
is obtained by picking at random (with equal chances) one of the Schl\"afli cones induced by $\Hc_1,\dots,\Hc_n$.
\end{definition}

Since consecutive random constructions, of which this is an example, will also appear later, we indicate, once and for all, how such a procedure can be formalized. Let $\Omega_1^n:= G(d,d-1)_*^n$ be the set of $n$-tuples of $(d-1)$-subspaces in general position. The probability measure $P_n$ on $\Omega_1^n$ is defined by $P_n:=\phi^{*n}\fed \Omega_1^n$ (where $\fed$ denotes the restriction of a measure). We interpret the choice described in Definition \ref{D3.2} as a two-step experiment and define a kernel $ K_2^1: \Omega_1^n\times{\mathcal B}(\mathcal {PC}^d)\to[0,1]$ by
$$ K_2^1(\eta_n,B):= \frac{1}{C(n,d)} \sum_{C\in {\mathcal F}_d(\eta_n)} {\bf 1}_{B}(C)$$
for $\eta_n\in \Omega_1^n$ and $B\in{\mathcal B}(\mathcal {PC}^d)$. Then (following, e.g., \cite[Satz 1.8.10]{GS77}), we define a probability measure $P_n\times K_2^1$ on ${\mathcal B}(\Omega_1^n) \otimes{\mathcal B}(\mathcal {PC}^d)$ by
\begin{eqnarray*}
(P_n\times K_2^1)(A)&=&\int_{G(d,d-1)^n}\int_{\mathcal {PC}^d} {\bf 1}_A(\eta_n, \omega_2) \, K_2^1(\eta_n, \D\omega_2)\, \phi^{*n}(\D \eta_n)\\
&=&\int_{G(d,d-1)^n} \frac{1}{C(n,d)} \sum_{C\in {\mathcal F}_d(\eta_n)}  {\bf 1}_A(\eta_n,C)\, \phi^{*n}(\D \eta_n)
\end{eqnarray*}
for $A\in {\mathcal B}(\Omega_1^n) \otimes{\mathcal B}(\mathcal {PC}^d)$. Now $S_n$ is defined as the random cone whose  distribution is equal to  $(P_n\times K_2^1)(\Omega_1^n\times\cdot)$. Thus, 
\begin{equation}\label{3.3}  
{\mathbb P}(S_n\in B) = \int_{G(d,d-1)^n} \frac{1}{C(n,d)}\sum_{C\in{\mathcal F}_d(H_1,\dots,H_n)} {\bf 1}_B(C) \, \phi^{*n}(\D (H_1,\dots,H_n))
\end{equation}
for $B\in{\mathcal B}(\mathcal {PC}^d)$. 

To relate $S_n$ and $C_n$, we rewrite equation (\ref{3.2a}), using the symmetry of $\phi$ and then (\ref{3.2}) and (\ref{3.1}). For $B\in{\mathcal B}(\mathcal{PC}^d_p)$, we obtain
\begin{eqnarray*} 
{\mathbb P}(C_n\in B) &=& \frac{1}{p_n^{(d)}} \int_{(\Sd)^n} \frac{1}{2^n}\sum_{\eps_i=\pm 1} {\bf 1}_B({\rm pos}\{\eps_1 x_1,\dots,\eps_n x_n\})\,\phi^n(\D(x_1,\dots,x_n))\\
&=&  \int_{(\Sd)^n} \frac{1}{C(n,d)}\sum_{\eps_i=\pm  1} {\bf 1}_B\left(\left(\bigcap_{i=1}^n\eps_ix_i^-\right)^\circ\right) \,\phi^n(\D (x_1,\dots,x_n))\\
&=& \int_{(\Sd)^n}  \frac{1}{C(n,d)}\sum_{C\in{\mathcal F}_d(x_1^\perp,\dots,x_n^\perp)} {\bf 1}_B(C^\circ) \,\phi^n(\D (x_1,\dots,x_n))\\
&=& \int_{G(d,d-1)^n}  \frac{1}{C(n,d)}\sum_{C\in{\mathcal F}_d(H_1,\dots,H_n)} {\bf 1}_B(C^\circ) \,\phi^{*n}(\D (H_1,\dots,H_n))\\
&=& {\mathbb P}(S_n^\circ\in B),
\end{eqnarray*}
where (\ref{3.3}) was used in the last step. Since also ${\mathbb P}(S_n^\circ={\mathbb R}^d) = {\mathbb P}(S_n=\emptyset)=0$, we can formulate the following.

\begin{theorem}\label{T3.1}
Let $\phi$ be an even probability measure on $\Sd$ which assigns measure zero to each $(d-2)$-dimensional great subsphere, let $n\in{\mathbb N}$. Then the $(\phi,n)$-Cover--Efron cone $C_n$ and the dual of the $(\phi^*,n)$-Schl\"afli cone, $S_n^\circ$, are stochastically equivalent,
\begin{equation}\label{3.10}
C_n=S_n^\circ\quad\mbox{in distribution.}
\end{equation}
\end{theorem}

\section{Expectations for random Schl\"afli and Cover--Efron cones}\label{sec4}

In this section, $\phi^*$ is a probability measure on the Grassmannian $G(d,d-1)$ with the property that it is zero on each set of hyperplanes containing a fixed line through $0$. For $n\in{\mathbb N}$, we consider the $(\phi^*,n)$-Schl\"afli cone and want to compute the expectations of the geometric functionals $Y_{k,j}$, defined by (\ref{2.5}), for this random cone. 

In his study of Poisson hyperplane tessellations in Euclidean spaces, Miles \cite[Chap. 11]{Mil61} has employed the idea of defining, by means of combinatorial selection procedures, different weighted random polytopes, which could then be combined to give results about first and second moments. In this and subsequent sections, we adapt this approach to conical tessellations.

First we describe a combinatorial random choice. Let $H_1,\dots,H_n \in G(d,d-1)$ be hyperplanes in general position, and let $L\in G(d,k)$, for $k\in \{1,\dots,d\}$, be a $k$-dimensional linear subspace in general position with respect to $H_1,\dots,H_n$, which means that $H_1\cap L,\dots,H_n\cap L$ are $(k-1)$-dimensional 
subspaces of $L$ which are in general position in $L$. Let $j\in\{1,\dots,k\}$.  The tessellation ${\mathcal T}_L$ induced in $L$ by $H_1\cap L,\dots,H_n\cap L$, has $C(n,k,j)$ faces of dimension $j$, by (\ref{3.1a}). If $n<k-j$, then clearly $C(n,k,j)=0$. The following is an immediate consequence of general position.

\begin{lemma}
Let $j\ge 1$. To each $j$-face $F_j$ of ${\mathcal T}_L$, there is a unique $(d-k+j)$-face $F$ of the tessellation ${\mathcal T}$ induced by $H_1,\dots,H_n$, such that $F_j=F\cap L$.

Conversely, if $F\in {\mathcal F}_{d-k+j}({\mathcal T})$ and $F\cap L\not=\{0\}$, then $F\cap L$ is a $j$-face of ${\mathcal T}_L$.
\end{lemma}

In the following, we assume that $n\ge k-j$. We choose one of the $j$-faces of ${\mathcal T}_L$ at random (with equal chances) and denote it by $F_j$. Then $F_j=L\cap F$ with a unique face $F\in{\mathcal F}_{d-k+j}({\mathcal T})$. The face $F_j$ is contained in $2^{k-j}$ Schl\"afli cones of ${\mathcal T}_L$ and thus in $2^{k-j}$ Schl\"afli cones of ${\mathcal T}$. These are precisely the Schl\"afli cones of ${\mathcal T}$ that contain $F$. We select one of these at random (with equal chances) and call it $C^{[k,j]}(H_1,\dots,H_n,L)$.

Let $\Hc_1,\dots,\Hc_n$ be independent random hyperplanes with distribution $\phi^*$. We apply the described procedure to these hyperplanes and to a random $k$-dimensional subspace. This random subspace will here be chosen as explained below, and in a different way in Section \ref{sec6}.

Let $\mathcal L\in G(d,k)$ be a random subspace with distribution $\nu_k$, which is independent of $\Hc_1,\dots,\Hc_n$; for $k=d$, $L=\Rd$ is deterministic. We may assume, since this happens with probability $1$, that $\Hc_1,\dots,\Hc_n$ and $\mathcal L$ are in general position. Then we define \begin{equation}\label{4.0}
C_n^{[k,j]} := C^{[k,j]}(\Hc_1,\dots,\Hc_n,\mathcal L).
\end{equation}
More formally, $C_n^{[k,j]}$ is a random polyhedral cone with distribution given by
\begin{eqnarray}\label{3.13}
& & {\mathbb P}(C_n^{[k,j]} \in B)\\
&& = \int_{G(d,d-1)^n} \int_{G(d,k)} \frac{1}{C(n,k,j)} \sum_{F\in{\mathcal F}_{d-k+j}(\eta_n)\atop F\cap L \not=\{0\}}
\frac{1}{2^{k-j}}\sum_{C\in{\mathcal F}_{d}(\eta_n) \atop C\supset F} {\bf 1}_B(C) \,\nu_k(\D L) \,\phi^{*n}(\D\eta_n)  \nonumber
\end{eqnarray}
for $B\in{\mathcal B}(\mathcal {PC}^d)$ and $n\ge k-j$ (recall that $\eta_n$ is a shorthand notation for $(H_1,\dots,H_n)$). 

If $n>k-j$, then almost surely $F\in{\mathcal F}_{d-k+j}(\eta_n)$ is not a linear subspace. 
Thus, \eqref{sub1} implies that the inner integral in (\ref{3.13}), up to the combinatorial factors, can be written as 
\begin{eqnarray*}
&& \int_{G(d,k)} \sum_{F\in{\mathcal F}_{d-k+j}(\eta_n)} {\bf 1}\{F\cap L \not=\{0\}\}  \sum_{C\in{\mathcal F}_{d}(\eta_n)} {\bf 1}\{F\subset C\}  {\bf 1}_B(C) \,\nu_k(\D L)\\
&& = \sum_{C\in{\mathcal F}_{d}(\eta_n)} {\bf 1}_B(C)  \sum_{F\in{\mathcal F}_{d-k+j}(\eta_n)} {\bf 1}\{F\subset C\} \int_{G(d,k)} {\bf 1}\{F\cap L \not=\{0\}\} \,\nu_k(\D L)
\end{eqnarray*} 
\begin{eqnarray*}
&& = \sum_{C\in{\mathcal F}_{d}(\eta_n)} {\bf 1}_B(C)  \sum_{F\in{\mathcal F}_{d-k+j}(\eta_n)} {\bf 1}\{F\subset C\} 2 U_{d-k}(F)\\
&& = 2\sum_{C\in{\mathcal F}_{d}(\eta_n)} {\bf 1}_B(C)  Y_{d-k+j,d-k}(C),
\end{eqnarray*}
according to (\ref{2.5}). Therefore, we obtain
\begin{equation}\label{3.14}
{\mathbb P}(C_n^{[k,j]} \in B) = \frac{2}{2^{k-j}C(n,k,j)} \int_{G(d,d-1)^n}  \sum_{C\in{\mathcal F}_{d}(\eta_n)} {\bf 1}_B(C)  Y_{d-k+j,d-k}(C) \,\phi^{*n}(\D\eta_n).
\end{equation} 

From (\ref{3.3}) and (\ref{3.14}) (both formulated for expectations) we get, for every nonnegative, measurable function $g$ on $\mathcal {PC}^d$ and $n> k-j$, the equation
\begin{equation}\label{3.15}
{\mathbb E}\,g(C_n^{[k,j]}) = \frac{2C(n,d)}{2^{k-j}C(n,k,j)} {\mathbb E}\,(gY_{d-k+j,d-k})(S_n).
\end{equation}
Choosing $g=1$ in (\ref{3.15}), we obtain the following theorem.

\begin{theorem}\label{T4.0} 
The expected size functionals ${\mathbb E}\,Y_{i,j}$ of the $(\phi^*,n)$-Schl\"afli cone $S_n$ are given by
\begin{equation}\label{3.16n} 
{\mathbb E}\,Y_{d-k+j,d-k}(S_n) = \frac{2^{k-j}C(n,k,j)}{2C(n,d)} ,
\end{equation}
for $1\le j\le k\le d$ and $n> k-j$.
\end{theorem}

As a consequence, we can also write
$$
{\mathbb E}\,g(C_n^{[k,j]}) = \frac{{\mathbb E}\,(gY_{d-k+j,d-k})(S_n)}{{\mathbb E}\,Y_{d-k+j,d-k}(S_n)}.
$$
Thus, the distribution of $C_n^{[k,j]}$ is obtained from the distribution of $S_n$ by weighting it with the function $Y_{d-k+j,d-k}$. This is the conical counterpart to 
\cite[Sec.~11.3, Lemma]{Mil61}. In analogy to \cite[Sec.~11.3]{Mil61}, we point out some special cases.

If $k=j=1$, the procedure described above is equivalent to choosing a uniform random point in $\Sd$, independent of $\Hc_1,\dots,\Hc_n$, and taking for $C_n^{[1,1]}$ the Schl\"afli cone containing it. The weight function satisfies $Y_{d,d-1}(C) = V_d(C)$.

If $k=d$, the procedure is equivalent to choosing a $j$-face of the tessellation ${\mathcal T}$ at random (with equal chances) and then choosing at random (with equal chances) one of the Schl\"afli cones containing it, which gives $C_n^{[d,j]}$. The weight function satisfies $Y_{j,0}(C) =\frac{1}{2} f_j(C)$, since the assumption $n>d-j$ implies that the $j$-faces of $C$ are not linear subspaces. In particular, for $j=d$ it is constant, and $C_n^{[d,d]}=S_n$ in distribution.

By specialization, the equation (\ref{3.16n}) includes the following results, which were obtained by Cover and Efron \cite{CE67}. 

\begin{corollary}\label{C1}
For $k=1,\dots,d$,  
\begin{equation}\label{C1.1}
{\mathbb E} f_k(S_n) = \frac{2^{d-k}\binom{n}{d-k}C(n-d+k,k)}{C(n,d)},
\end{equation}
and for $k=0,\dots,d-1$,
\begin{equation}\label{C1.2}
{\mathbb E} f_k(C_n)=\frac{2^k\binom{n}{k}C(n-k,d-k)}{C(n,d)}.
\end{equation}
\end{corollary}

Equation (\ref{C1.1}) is formula (3.1) in \cite{CE67}, after correction of misprints. This equation is obtained from (\ref{3.16n}) by choosing $k=d$ and then replacing $j$ by $k$ (and observing (\ref{2.5b}) and (\ref{3.1a})), if $n> d-k$. For $n=d-k$, both sides are equal to $1$, and for $n<d-k$ both sides are zero.  The duality (\ref{3.10}) gives (\ref{C1.2}), which is formula (3.3) in \cite{CE67}. 

The following expectations do not appear in \cite{CE67}. 

\begin{corollary}\label{C3}
The expected conical quermassintegrals of the $(\phi^*,n)$-Schl\"afli cone $S_n$ and the $(\phi,n)$-Cover--Efron cone $C_n$ are given by
\begin{equation}\label{4.30}
{\mathbb E}\,U_k(S_n) = \frac{C(n,d-k)}{2C(n,d)}
\end{equation}
for $ k=0,\dots, d-1$, and by
\begin{equation}\label{4.31}
{\mathbb E}\,U_k(C_n) = \frac{C(n,d)-C(n,k)}{2C(n,d)}.
\end{equation}
for $k=1,\dots,d-1$.
\end{corollary}

Equation (\ref{4.30}) is obtained by replacing $k$ and $j$ in (\ref{3.16n}) both by $d-k$. Note that if $n\le d-k$, then both sides of the equation are equal to $1/2$. Since $C_n$ is almost surely pointed, the dualities (\ref{2.4}) and (\ref{3.10}) yield (\ref{4.31}), where both sides of the equation are equal to $0$ if $n<k$. 

We can now apply (\ref{2.3}) for $j=1,\dots,d$ together with (\ref{4.30}), and (\ref{2.1}) for $j=0$ together with  (\ref{3.10}) and (\ref{4.31}), to obtain (\ref{4.30a}) below. The duality relations (\ref{2.1}) and (\ref{3.10}) then yield (\ref{4.30b}).
\begin{corollary}\label{T4.1}

\begin{equation}\label{4.30a}
{\mathbb E}\,V_j(S_n) = \left\{ \begin{array}{ll} \displaystyle\binom{n}{d-j}C(n,d)^{-1},& j=1,\dots,d,\\[5mm] \displaystyle\binom{n-1}{d-1}C(n,d)^{-1}, & j= 0.\end{array}\right.
\end{equation}
and
\begin{equation}\label{4.30b}
{\mathbb E}\,V_j(C_n) = \left\{ \begin{array}{ll} \displaystyle\binom{n}{j}C(n,d)^{-1},& j=0,\dots,d-1,\\[5mm] \displaystyle\binom{n-1}{d-1}C(n,d)^{-1}, & j= d.\end{array}\right.
\end{equation}
\end{corollary}

\noindent{\bf Remark.} After a first version of this manuscript had been posted in the arXiv, Martin Lotz kindly pointed out to the authors that relation (\ref{4.30a}) can also be deduced from a result of Klivans and Swartz \cite{KS11}, for which he sketched a simpler proof. Let ${\mathcal A}$ be an arrangement of $n$ hyperplanes through $0$ in $\Rd$. The main result of \cite{KS11} connects the polynomial $\sum_{k=0}^d \sum V_k(C)t^k$, where the inner sum extends over the $d$-cones of the tessellation induced by ${\mathcal A}$, with the characteristic polynomial of ${\mathcal A}$ and thus with the M\"obius function of the intersection poset of ${\mathcal A}$. Under our assumption of general position, this M\"obius function is easily determined, therefore the result of \cite{KS11} yields (\ref{4.30a}) (though with a less direct proof). Meanwhile, a short proof of the Klivans--Swartz formula has independently been given by Kabluchko, Vysotsky and Zaporozhets in \cite[Theorem 4.1]{KVZ15}, and Amelunxen and Lotz \cite[Theorem 6.1]{AL15} have generalized that formula to faces of all dimensions.

In the summary of their paper \cite{CE67}, Cover and Efron also announced results on the `expected natural measure of the set of $k$-faces'. As such a natural measure one can consider the total $k$-face content $\Lambda_k$ defined by (\ref{2.6}) for polyhedral cones (or its natural analogue in the case of spherical polytopes). The following can be stated.

\begin{proposition}
For the functionals defined by $\Lambda_k(C)=\sum_{F\in{\mathcal F}_k(C)} V_k(F)$, the expectations for random Schl\"afli cones are given by
\begin{equation}\label{3.7a} 
{\mathbb E}\,\Lambda_k(S_n) = \frac{2^{d-k}\binom{n}{d-k}}{C(n,d)},
\end{equation}
for $k=1,\dots,d$, 
and for Cover--Efron cones by 
\begin{equation}\label{3.7b} 
{\mathbb E}\,\Lambda_k(C_n) = \frac{\binom{n}{k}C(n-k,d-k)}{C(n,d)},
\end{equation}
for $k=1,\dots,d-1$. 
\end{proposition}

In contrast to (\ref{3.7b}), relation (\ref{3.7a}) holds also for $k=d$, by (\ref{3.16n}). Cover and Efron did not formulate these results; however, some arguments leading to them are contained in the proofs of their Theorems 2 and 4. We note that (\ref{3.7a}) is the special case of (\ref{3.16n}) which is obtained by replacing $k$ by $d-k+1$ and setting $j=1$. Here we use that for $n>d-k$, the $k$-faces of $S_n$ are not in $G(d,k)$. For $n\le d-k$, the equation is apparently true as well.

For (\ref{3.7b}), we extend and complete the arguments given in \cite{CE67}. For the proof, we can assume that $n\ge k$. Let $k\in\{1,\dots,d-1\}$. By (\ref{3.10}) and (\ref{3.3}),
$$ {\mathbb E}\,\Lambda_k(C_n) = {\mathbb E}\,\Lambda_k(S_n^\circ) =\int_{G(d,d-1)^n} \frac{1}{C(n,d)} \sum_{C\in{\mathcal F}_d(\eta_n)} \Lambda_k(C^\circ)\,\phi^{*n}(\D\eta_n).$$

Let $\eta_n=(H_1,\dots,H_n)$, where $H_1,\dots,H_n\in G(d,d-1)$ are in general position. Let $F\in{\mathcal F}_{d-k}(\eta_n)$. Then there are indices $1\le i_1<\dots< i_k\le n$ such that 
$$ F\subset L_{i_1,\dots,i_k}:= H_{i_1}\cap\dots\cap H_{i_k}.$$
Let ${\mathcal C}_F$ be the set of Schl\"afli cones $C\in{\mathcal F}_d(\eta_n)$ with $F\subset C$.
Let $u_j$ be a unit normal vector of $H_{i_j}$, $j=1,\dots,k$. Then the cones $C\in{\mathcal C}_F$ are in one-to-one correspondence with the choices $\eps_1,\dots,\eps_k\in\{-1,1\}$ such that
$$ C\subset \bigcap_{j=1}^k \eps_j u_j^-.$$
The face of $C^\circ$ conjugate to $F$ (with respect to $C$) is then given by 
$$ \widehat F_C ={\rm pos}\{\eps_1u_1,\dots,\eps_k u_k\}.$$
It follows that the faces $\widehat F_C$, $C\in{\mathcal C}_F$, form a tiling of $L_{i_1,\dots,i_k}^\perp$, and therefore
\begin{equation}\label{CE1}
\sum_{C\in{\mathcal F}_d(\eta_n)} {\bf 1}\{F\subset C\}V_k(\widehat F_C)=1.
\end{equation}
The faces $F\in{\mathcal F}_{d-k}(\eta_n)$ with $F\subset L_{i_1,\dots,i_k}$ are the Schl\"afli cones of the tessellation induced in $L_{i_1,\dots,i_k}$, hence there are precisely $C(n-k,d-k)$ of them. Now we obtain, using (\ref{CE1}) and the latter remark,
\begin{align*}
 \sum_{C\in{\mathcal F}_d(\eta_n)} \Lambda_k(C^\circ)
&=\sum_{C\in{\mathcal F}_d(\eta_n)}\;\sum_{G\in{\mathcal F}_k (C^\circ)} V_k(G)= \sum_{C\in{\mathcal F}_d(\eta_n)}\; \sum_{F\in{\mathcal F}_{d-k}(C)} V_k(\widehat F_C)\\
&=\sum_{F\in{\mathcal F}_{d-k}(\eta_n)}\; \sum_{C\in{\mathcal F}_d(\eta_n)} {\bf 1}\{F\subset C\}V_k(\widehat F_C)\\
&= \sum_{1\le i_1 <\dots< i_k\le n} \, \sum_{F\in{\mathcal F}_{d-k}(\eta_n)} {\bf 1}\{F\subset L_{i_1,\dots,i_k}\}  \sum_{C\in{\mathcal F}_d(\eta_n)} {\bf 1}\{F\subset C\}V_k(\widehat F_C)\\
&= \binom{n}{k} C(n-k,d-k),
\end{align*}
which yields (\ref{3.7b}).

We point out that the results obtained so far hold for general distributions $\phi^*$, as specified at the beginning of this section (which exhibits their essentially combinatorial character).

\section{Some first and second order moments}\label{sec5}

We have defined the random Schl\"afli cone by picking at random, with equal chances, one of the $d$-cones generated by a finite number of i.i.d. random hyperplanes through $0$ (with a suitable distribution). A different model of a random cone is obtained by taking the (almost surely unique) cone that contains a fixed given ray. This is in analogy to the Euclidean case, where, for a stationary random mosaic, the typical cell and the zero cell (containing the origin) are classical examples of random polytopes. In that case, it is known (e.g., \cite[Thm. 10.4.1]{SW08}) that the distribution of the zero cell is, up to translations, the volume-weighted distribution of the typical cell. In this section, we derive an analogous statement for conical tessellations generated by hyperplanes with rotation invariant distribution (Lemma \ref{L5.2}), and also some expectation results in analogy to the Euclidean case. While this is of independent interest, our main goal is to derive from this, together with the expectation (\ref{5.4}), the mixed second moment (\ref{5.6}), because this is an essential prerequisite for the proof of our main result, Theorem \ref{T8.1}.

Recall that $\nu_{d-1}$ denotes the unique rotation invariant probability measure on the Grassmannian $G(d,d-1)$. The subsequent results require this special distribution for the considered random hyperplanes, instead of the general distribution $\phi^*$ of the previous sections. 

First we formulate a simple lemma. 

\begin{lemma}\label{L5.1}If $A\in{\mathcal B}(\Sd)$ and $k\in\{1,\dots,d-1\}$, then
\begin{equation}\label{5.1} 
\int_{G(d,d-1)^k} \sigma_{d-k-1}(A\cap H_1\cap\dots\cap H_k)\,\nu_{d-1}^k(\D(H_1,\dots,H_k)) =\frac{\omega_{d-k}}{\omega_d} \sigma_{d-1}(A).
\end{equation}
\end{lemma}

\begin{proof}
As a function of $A$, the left-hand side of (\ref{5.1}) is a finite measure, which, due to the rotation invariance of $\nu_{d-1}$ and of $\sigma_{d-k-1}$, must be invariant under rotations. Up to a constant factor, there is only one such measure on ${\mathcal B} (\Sd)$, namely $\sigma_{d-1}$. The choice $A=\Sd$ then reveals the factor.
\end{proof}

Now let $\Hc_1,\dots,\Hc_n$  be independent random hyperplanes through $0$ with distribution $\nu_{d-1}$. Before treating the $(\nu_{d-1},n)$-Schl\"afli cone, we consider a different random cone, which corresponds to the zero cell in the theory of Euclidean tessellations. Let $e\in\Sd$ be a fixed vector. With probability $1$, the vector $e$ is contained in a unique Schl\"afli cone induced by $\Hc_1,\dots,\Hc_n$, and we denote this cone by $S_n^e$. If $e\notin H\in G(d,d-1)$, we denote by $H^e$ the closed halfspace bounded by $H$ that contains $e$.

Let $k\in\{0,\dots,d-1\}$. Almost surely, each $(d-k)$-face of $S_n^e$ is the intersection of $S_n^e$ with exactly $k$ of the hyperplanes $\Hc_1,\dots,\Hc_n$. Conversely, each intersection of $k$ distinct hyperplanes from $\Hc_1,\dots,\Hc_n$ a.s. intersects $S_n^e$ either in a $(d-k)$-face or in $\{0\}$. Observing this, we compute
\begin{align*}
\E\Lambda_{d-k}(S_n^e)
&  = \E \sum_{1\le i_1<\dots<i_k \le n} V_{d-k}(S_n^e \cap \Hc_{i_1}\cap\dots\cap \Hc_{i_k})\\
&  = \sum_{1\le i_1<\dots<i_k \le n} \E V_{d-k}(\Hc_1^e\cap\dots\cap \Hc_n^e \cap \Hc_{i_1}\cap\dots\cap \Hc_{i_k})\\
&  = \binom{n}{k} \E V_{d-k}(\Hc_{k+1}^e\cap\dots\cap\Hc_n^e  \cap \Hc_1\cap\dots\cap \Hc_k)\\
&  = \binom{n}{k} \int_{G(d,d-1)^{n-k}} \int_{G(d,d-1)^k}V_{d-k}(H_{k+1}^e\cap\dots\cap H_n^e \cap H_1 \cap\dots\cap H_k)\\
& \qquad \times \,\nu_{d-1}^k(\D(H_1,\dots H_k))\,\nu_{d-1}^{n-k}(\D(H_{k+1},\dots,H_n)).
\end{align*}
If $n=k$, the outer integration does not appear, and  $\Hc_{k+1}^e\cap\dots\cap\Hc_n^e$ has to be interpreted as $\R^d$. For $n<k$, both sides of the equation are zero. 

By Lemma \ref{L5.1}, the inner integral is equal to
$$ \frac{1}{\omega_d}\sigma_{d-1}(H_{k+1}^e\cap\dots\cap H_n^e \cap\Sdԩ,$$
hence we obtain 
\begin{equation}\label{5.2} 
\E \Lambda_{d-k}(S_n^e) = \binom{n}{k} \E V_d(S_{n-k}^e)
\end{equation}
for $k=0,\ldots,d-1$. Here both sides of the equation are zero if $n<k$, and they are equal to $1$ for $n=k$. 

We next derive a similar formula for $\E f_{d-k}(S_n^e)$ (in analogy to \cite[Sec.~5]{Sch09}). Let $k\in\{0,\dots,d-1\}$ and $n>k$. 
As above, we obtain
\begin{eqnarray*}
\E f_{d-k}(S_n^e)
&=& \E \sum_{1 \le i_1<\dots< i_k\le n} {\bf 1}\{S_n^e \cap \Hc_{i_1}\cap\dots\cap\Hc_{i_k}\not=\{0\}\}\\
& = &\binom{n}{k} \int_{G(d,d-1)^{n-k}} \int_{G(d,d-1)^k} {\bf 1}\{H_{k+1}^e\cap\dots\cap H_n^e \cap H_1 \cap\dots\cap H_k \not=\{0\}\}\\
& &\times \,\nu_{d-1}^k(\D(H_1,\dots H_k))\,\nu_{d-1}^{n-k}(\D(H_{k+1},\dots,H_n)).
\end{eqnarray*}
Let $G(d,d-1)^k_*$ denote the set of all $k$-tuples of $(d-1)$-dimensional linear subspaces with linearly independent normal vectors. 
The image measure of $\nu_{d-1}^k$ under the mapping $(H_1,\dots,H_k)\mapsto H_1\cap\dots\cap H_k$ from $G(d,d-1)^k_*$ to $G(d,d-k)$ is the invariant measure $\nu_k$, hence
$$ \int_{G(d,d-1)^k} {\bf 1}\{C \cap H_1 \cap\dots\cap H_k \not=\{0\}\}\,\nu_{d-1}^k(\D(H_1,\dots H_k)) =2U_k(C)$$
for  $C=H_{k+1}^e\cap\dots\cap H_n^e\in{\mathcal C}^d$ and $\nu_{d-1}^{n-k}$ almost all $(H_{k+1},\dots,H_n)\in G(d,d-1)^{n-k}$.  
We conclude that
$$
\E f_{d-k}(S_n^e) = 2\binom{n}{k} \E U_k(S_{n-k}^e)
$$
for  $k\in\{0,\ldots,d-1\}$ and $n>k$. If $n=k$, then $\E f_{d-k}(S_n^e)=1$, and the expectation is zero for $n<k$.

To compute $\E V_d(S^e_n)$, let $P\subset\Sd$ be a closed spherically convex set containing $e$. Writing $u\in\Sd$ in the form $u=te+\sqrt{1-t^2}\,\overline u$ with $\overline u\in e^\perp\cap\Sd$, we have
\begin{equation}\label{5.3}  
\sigma_{d-1}(P)=\int_{e^\perp\cap\hspace{1pt}\Sd} \int_{\cos \rho(P,\overline u)}^1 (1-t^2)^{\frac{d-3}{2}} \,\D t\, \sigma_{d-2}(\D \overline u)
\end{equation}
with
$$ \rho(P,\overline u) = \max\{\rho\in [0,\pi]:(\cos\rho)e+(\sin\rho)\overline u\in P\},\quad \overline u\in e^\perp\cap\Sd.$$
Let $Z_n^e:= S_n^e\cap\Sd$. For fixed $\overline u\in e^\perp\cap\Sd$, the distribution function of the random variable $\rho(Z_n^e,\overline u)$ is given by
$$ F(x) = {\mathbb P}\left(\rho(Z_n^e,\overline u)< x\right) = 1-\left(1-\frac{x}{\pi}\right)^n,$$
since $\rho(Z_n^e,\overline u) >x$ holds if and only if none of the hyperplanes $\Hc_1,\dots,\Hc_n$ intersects the great circular arc connecting $e$ and $(\cos x)e+(\sin x)\overline u$. Let
$$ G(x):= \int_{\cos x}^1 (1-t^2)^{\frac{d-3}{2}} \,\D t = \int_0^x \sin^{d-2}\alpha\,\D\alpha\qquad \mbox{for }x\in [0,\pi].$$
From (\ref{5.3}) we have $G(\pi)=\omega_d/\omega_{d-1}$. Since the distribution of the random variable  $\rho(Z_n^e, \overline u)$ does not depend on $\overline u$, we obtain
\begin{eqnarray*}
\E\sigma_{d-1}(Z_n^e) &=& \E \int_{e^\perp\cap\Sd} \int_{\cos \rho(Z_n^e,\overline u)}^1 (1-t^2)^{\frac{d-3}{2}} \D t\,\sigma_{d-2}(\D \overline u)\\
&=& \omega_{d-1} \E G(\rho(Z_n^e,\overline u))\\
&=& \omega_{d-1} \int_0^\pi G(x)F'(x)\,\D x\\
&=& \omega_{d-1}\left[ G(\pi) -\int_0^\pi G'(x)F(x)\,\D x\right]\\
&=& \omega_{d-1}\left[ \frac{\omega_d}{\omega_{d-1}} -\int_0^\pi \sin^{d-2}x\left(1-\left(1-\frac{x}{\pi}\right)^n\right)\D x\right]\\
&=& \omega_{d-1} \int_0^\pi\left(1-\frac{x}{\pi}\right)^n\sin^{d-2}x\,\D x.
\end{eqnarray*}
After using the binomial theorem, the integral can be evaluated by using recursion formulas and known definite integrals; e.g., see \cite[p.~117]{GH50}. (The evaluation of the integral for $d=3$ in \cite[(6.16)]{Mil71} is corrected in \cite{CM09}.)

Defining the constant $\theta(n,d)$ by
\begin{equation}\label{varpi} 
\theta(n,d) := \frac{\omega_{d-1}}{\omega_d} \int_0^\pi\left(1-\frac{x}{\pi}\right)^n\sin^{d-2}x\,\D x,\qquad \text{for } n\in \mathbb{N}_0,
\end{equation}
and by $\theta(n,d):=0$ for $n<0$, 
and recalling that $V_d(S_n^e)=\sigma_{d-1}(Z_n^e)/ \omega_d$, we can write the result as 
\begin{equation}\label{5.4}
\E V_d(S_n^e)= \theta(n,d).
\end{equation}
Note that $\theta(0,d)=1$. 
As a corollary, we obtain from (\ref{5.2}) that
\begin{equation}\label{5.4a}
\E\Lambda_{d-k}(S_n^e)= \binom{n}{k}\theta(n-k,d)
\end{equation}
for $k\in\{0,\ldots,d-1\}$. For $n=k$ both sides are equal to $1$, and they are zero for $n<k$. 

The following lemma relates the distribution of $S_n^e$ to that of the random $(\nu_{d-1},n)$-Schl\"afli cone $S_n$.

\begin{lemma}\label{L5.2}
Let $ \Hc_1,\dots,\Hc_n$ be independent random hyperplanes with distribution $\nu_{d-1}$, and let $S_n^e$ be the induced Schl\"afli cone containing the fixed given vector $e\in \Sd$.

Let $f$ be a nonnegative measurable function on $\mathcal {PC}^d$ which is invariant under rotations. Then
$$ \E f(S_n^e) =C(n,d) \,\E(fV_d)(S_n).$$
\end{lemma}

\begin{proof}
In the following, we denote by $\nu$ the invariant probability measure on the rotation group ${\rm SO}(d)$, and we make use of the fact that
$$ \int_{{\rm SO}(d)} g(\vartheta e)\,\nu(\D \vartheta) = \frac{1}{\omega_d} \int_{\Sd} g(u)\,\sigma_{d-1}(\D u)$$
for every nonnegative measurable function $g$ on $\Sd$. Using the rotation invariance of the function $f$ and of the probability distribution $\nu_{d-1}$, we obtain, with $\vartheta\in{\rm SO}(d)$,
\begin{eqnarray*}
\E f(S_n^e) & = &\E\sum_{C\in{\mathcal F}_d(\Hc_1,\dots,\Hc_n)} f(C){\bf 1}_{{\rm int}\,C}(e)\\
& = & \E\sum_{C\in{\mathcal F}_d(\Hc_1,\dots,\Hc_n)} f(C){\bf 1}_{{\rm int}\,C}(\vartheta e)\\
& = &\E\int_{{\rm SO}(d)} \sum_{C\in{\mathcal F}_d(\Hc_1,\dots,\Hc_n)} f(C){\bf 1}_{{\rm int}\,C}(\vartheta e)\,\nu(\D\vartheta)\\ 
& = &\frac{1}{\omega_d} \,\E\int_{\Sd} \sum_{C\in{\mathcal F}_d(\Hc_1,\dots,\Hc_n)} f(C){\bf 1}_{{\rm int}\,C}(u)\,\sigma_{d-1}(\D u)\\
& = &\frac{1}{\omega_d} \,\E \sum_{C\in{\mathcal F}_d(\Hc_1,\dots,\Hc_n)} f(C)\sigma_{d-1}(C\cap\Sd)\\
& =  & C(n,d) \,\E(fV_d)(S_n)
\end{eqnarray*}
by (\ref{3.3}) (with $\phi^*=\nu_{d-1}$).
\end{proof}

From Lemma \ref{L5.2} and (\ref{5.4a}) we get
\begin{equation}\label{5.6}
\E(\Lambda_{d-k}V_d)(S_n) = \frac{\binom{n}{k}\theta(n-k,d)}{C(n,d)}
\end{equation}
for $k=0,\dots,d-1$. The case $k=0$ reads
$$
\E V_d^2(S_n) = \frac{\theta(n,d)}{C(n,d)}.
$$

Equation (\ref{5.6}) is a conical counterpart to Miles \cite[Thm.~11.1.1]{Mil61}. The special case $d=3$ of (\ref{5.6}) is contained in Miles \cite[Thm.~6.3]{Mil71}.

\section{Another selection procedure}\label{sec6}

In this section, we begin with the proof of our main result, Theorem \ref{T8.1}, which will yield all the mixed moments $\E(\Lambda_s\Lambda_r)(S_n)$. Before that, we sketch the proof strategy. The principal idea can already be seen from the way the mixed second moment (\ref{5.6}) for the random Schl\"afli cone $S_n$ was obtained. We had defined another random cone, $S_n^e$, with the property (expressed in Lemma 5.2) that its distribution is the $V_d$-weighted distribution of $S_n$. Since the expectation of $\E\Lambda_{d-k}(S_n^e)$ (see (\ref {5.4a})) could be determined by a direct geometric argument, we thus obtained the expectation $\E(\Lambda_{d-k}V_d)(S_n)$.

A more sophisticated version of this argument will finally allow us to determine explicitly the mixed moments $\E(\Lambda_s\Lambda_r)(S_n)$. In the present section, we use successive random choices to define a random cone $D_n^{[k,j]}$, for which we show in (\ref{3.16}) that its distribution is the $Y_{d-k+j,d-k}$-weighted distribution of $S_{n-d+k}$. The expectation $\E \Lambda_r (D_n^{[k,j]})$ is expressed in (\ref{8.3}) in terms of expectations for certain Schl\"afli cones. To obtain this, a geometric decomposition argument is needed, which is provided in Section \ref{sec7}. Both results together yield the expectation $\E(\Lambda_r Y_{d-k+j,d-k})(S_{n-d+k})$, which we can specialize and simplify to obtain $\E(\Lambda_s\Lambda_r)(S_n)$.

In Section \ref{sec4}, we have used a selection procedure to define a random cone $C_n^{[k,j]}$. This selection procedure will now be modified. The assumptions are the same as in Section \ref{sec5}: $\Hc_1,\dots,\Hc_n$ are independent random hyperplanes through $0$, each with distribution $\nu_{d-1}$, the rotation invariant probability measure on $G(d,d-1)$. 

The second selection procedure is equivalent to a conical analogue of the one in \cite[Sec.~11.4]{Mil61}, though we describe it in a different way. We assume again that $ 1\le j\le k\le d$ and $n\ge d-j$ (that is, $n-(d-k)\ge k-j$). Now a subspace ${\mathcal L}\in G(d,k)$ is chosen at random (with equal chances) from the $k$-dimensional intersections of the hyperplanes  $\Hc_1, \dots, \Hc_n$. (If $k=d$, then ${\mathcal L}=\Rd$ is deterministic. Corresponding adjustments can be made below.) There are indices $i_1,\dots,i_{d-k}\in\{1,\dots,n\}$ such that 
$${\mathcal L} = \Hc_{i_1} \cap \dots\cap \Hc_{i_{d-k}},$$
since $n\ge d-j\ge d-k$.   
In the following, if $\eta_n=(H_1,\dots,H_n)$, we denote by $\eta_n\langle i_1,\dots, i_{d-k} \rangle$ the $(n-d+k)$-tuple that remains when $H_{i_1},\dots,H_{i_{d-k}}$ have been removed from $(H_1,\dots,H_n)$. Similarly, ${\sf H}_n\langle i_1,\dots,i_{d-k}\rangle$ is obtained from ${\sf H}_n=(\Hc_1,\dots,\Hc_n)$. Then, employing the definition (\ref{4.0}), we define
$$ D_n^{[k,j]}:= C^{[k,j]}({\sf H}_n\langle i_1,\dots,i_{d-k}\rangle, {\mathcal L}).$$
(Note that the indices $i_1,\dots,i_{d-k}$ are determined by ${\mathcal L}$.)

Let $B\in{\mathcal B}(\mathcal {PC}^d)$. According to the definition of $D_n^{[k,j]}$, we have
\begin{align*}
{\mathbb P}(D_n^{[k,j]}\in B)
&= \int_{G(d,d-1)^n}\frac{1}{\binom{n}{d-k}} \sum_{1\le i_1<\dots< i_{d-k}\le n} 
\frac{1}{C(n-d+k,k,j)} \sum_{F\in{\mathcal F}_{d-k+j}(\eta_n\langle i_1,\dots,i_{d-k}\rangle) \atop F\cap H_{i_1}\cap\dots\cap H_{i_{d-k}} \not=\{0\}}\\
& \qquad\times\; \frac{1}{2^{k-j}} \sum_{C\in{\mathcal F}_{d}(\eta_n\langle i_1,\dots,i_{d-k}\rangle) \atop  C\supset F} {\bf 1}_B(C)\,\nu_{d-1}^n(\D \eta_n).
\end{align*}
For $k=d$, the condition $F\cap H_{i_1}\cap\dots\cap H_{i_{d-k}} \not=\{0\}$ is empty and can be deleted. Moreover, 
if $n=d-k$, then $j=k$, $F=C=\R^d$ and $D_{n}^{[k,j]}=D_{d-k}^{[k,k]}=\R^d$ almost surely. After interchanging the integration and 
the first summation, the summands of the sum $\sum_{1\le i_1<\dots< i_{d-k}\le n}$ 
are all the same. Therefore, we obtain
\begin{eqnarray}
&& {\mathbb P}(D_n^{[k,j]}\in B)\nonumber\\
&&= \frac{1}{2^{k-j}C(n-d+k,k,j)} \int_{G(d,d-1)^n}  \sum_{F\in{\mathcal F}_{d-k+j}(\eta_n\langle 1,\dots,d-k\rangle) \atop F\cap H_1\cap\dots\cap H_{d-k}\not=\{0\}} 
 \sum_{C\in{\mathcal F}_{d}(\eta_n\langle 1,\dots,d-k\rangle) \atop C\supset F} {\bf 1}_B(C)\,\nu_{d-1}^n(\D \eta_n)\nonumber\\
&&= \frac{1}{2^{k-j}C(n-d+k,k,j)} \int_{G(d,d-1)^{n-d+k}} \int_{G(d,d-1)^{d-k}} \sum_{F\in{\mathcal F}_{d-k+j} (H_{d-k+1}, \dots, H_n) \atop F\cap H_1\cap\dots\cap H_{d-k}\not=\{0\}}\nonumber\\
&& \hspace*{4mm}\times\;\sum_{C\in{\mathcal F}_{d}(H_{d-k+1}, \dots,H_n) \atop C\supset F}  {\bf 1}_B(C)\,\nu_{d-1}^{d-k}(\D (H_1,\dots,H_{d-k}))\,\nu_{d-1}^{n -d+k}(\D (H_{d-k+1},\dots,H_n)).
\label{3.15c}
\end{eqnarray}
If $n=d-k$, then the outer integration is omitted and $F=C=\R^d$. We have split the $n$-fold integration, since the image measure of the measure $\nu_{d-1}^{d-k}$ under the map $(H_1,\dots,H_{d-k})\mapsto H_1\cap\dots\cap H_{d-k}$ from  
$G(d,d-1)_*^{d-k}$ to $G(d,k)$ is (for reasons of rotation invariance) the Haar probability measure $\nu_k$ on $G(d,k)$. Therefore, for the inner integral we obtain
\begin{eqnarray*}
&&\int_{G(d,d-1)^{d-k}} \sum_{F\in{\mathcal F}_{d-k+j}(H_{d-k+1},\dots,H_n) \atop F\cap H_1\cap\dots\cap H_{d-k} \not=\{0\}}\sum_{C\in{\mathcal F}_{d}(H_{d-k+1},\dots,H_n) \atop C\supset F} {\bf 1}_B(C)\,\nu_{d-1}^{d-k}(\D (H_1, \dots, H_{d-k}))\\
&&= \int_{G(d,k)} \sum_{F\in{\mathcal F}_{d-k+j}(H_{d-k+1},\dots,H_n)} {\bf 1}\{F\cap L\not =\{0\}\} 
\sum_{C\in {\mathcal F}_{d}(H_{d-k+1}, \dots,H_n) \atop C\supset F}{\bf 1}_B(C)\,\nu_k(\D L).
\end{eqnarray*}
Assume now that $n>d-j$. Then, arguing as in the derivation of \eqref{3.14}, we see that the latter is equal to
\begin{eqnarray*}
&& \sum_{F\in{\mathcal F}_{d-k+j}(H_{d-k+1},\dots,H_n)} 2U_{d-k}(F) \sum_{C\in{\mathcal F}_{d} (H_{d-k+1},\dots,H_n)} {\bf 1}\{C\supset F)\}{\bf 1}_B(C)\\
&&=\sum_{C\in {\mathcal F}_{d}(H_{d-k+1},\dots,H_n)} \; \sum_{F\in{\mathcal F}_{d-k+j}(H_{d-k+1}, \dots,H_n)}{\bf 1}\{C\supset F\} 2U_{d-k}(F){\bf 1}_B(C)\\
&&= 2\sum_{C\in {\mathcal F}_{d}(H_{d-k+1},\dots,H_n)} Y_{d-k+j,d-k}(C) {\bf 1}_B(C).
\end{eqnarray*}
We conclude that
\begin{eqnarray*}
{\mathbb P}(D_n^{[k,j]}\in B) &=& \frac{2}{2^{k-j}C(n-d+k,k,j)} \int_{G(d,d-1)^{n-d+k}}\\
&&\times\, \sum_{C\in{\mathcal F}_{d} (\eta_{n-d+k})} {\bf 1}_B(C) Y_{d-k+j,d-k}(C) \,\nu_{d-1}^{n-d+k}(\D \eta_{n-d+k}).
\end{eqnarray*}
Together with (\ref{3.3}) (for $\phi^*=\nu_{d-1}$) this yields the following.
\begin{lemma} For every nonnegative, measurable function $g$ on $\mathcal {PC}^d$ and for $n>d-j$,
\begin{equation}\label{3.16}
\E g(D_n^{[k,j]}) = \frac{2C(n-d+k,d)}{2^{k-j}C(n-d+k,k,j)}\, {\mathbb E}\,(gY_{d-k+j,d-k})(S_{n-d+k}).
\end{equation}
\end{lemma}
As a consequencence, we have
$$
{\mathbb E}\,g(D_n^{[k,j]}) = \frac{{\mathbb E}\,(gY_{d-k+j,d-k})(S_{n-d+k})}{{\mathbb E}\,Y_{d-k+j,d-k}(S_{n-d+k})}.
$$
This is the conical counterpart to \cite[Thm. 11.5.1]{Mil61} (but in contrast to that, we have no equivalence here: $n$ on the left side and $n-d+k$ on the right side).

For later application, we note the special case $k=j$. From (\ref{3.15c}) and (\ref{3.16}) we obtain
\begin{eqnarray}\label{6.54}
& & \int_{G(d,d-1)^n}  \sum_{C\in{\mathcal F}_{d} (H_{d-j+1}, \dots, H_n) \atop C\cap H_1\cap\dots\cap H_{d-j}\not=\{0\}} g(C)\,\nu_{d-1}^n(\D(H_1,\dots,H_n)) = C(n-d+j,j)\,{\mathbb E}\,g(D_n^{[j,j]}) \nonumber\\
& & = 2C(n-d+j,d)\, {\mathbb E}\,(gU_{d-j})(S_{n-d+j})
\end{eqnarray}
for $n> d-j$.

\section{A geometric identity}\label{sec7}

To draw conclusions from the previous results, we need a geometric identity, given by (\ref{7.3}), in analogy to \cite[ Sec.~11.6]{Mil61}. Let $\eta_n=(H_1,\dots,H_n)\in G(d,d-1)^n_*$, let $j\in\{1,\dots,d-1\}$, $n> d-j$ and 
$$ L_j:= H_1\cap\dots\cap H_{d-j}.$$ 
Let $F_j\in{\mathcal F}_j(\eta_n)$ be a $j$-face such that $F_j\subset L_j$. Let $k\in\{j,\dots,d\}$. We delete the hyperplanes $H_{k-j+1},\dots,H_{d-j}$. From the tessellation induced by the remaining hyperplanes, we collect the $d$-cones containing $F_j$ and then classify their $r$-faces  for fixed $r$. Thus, we define
$$ {\mathcal F}_d(\eta_n,F_j,k)  := \{C\in{\mathcal F}_d(\eta_n\langle k-j+1,\dots,d-j\rangle): F_j\subset C\}.$$
Let $r\in\{1,\dots,d\}$. For $p\in{\mathbb N}$ with $r\le p\le d$ and $d-p\le k-j$, let 
\begin{eqnarray*}
&& {\mathcal F}_{r,p} :=\\
&& \big\{F\in{\mathcal F}_r(C): C\in {\mathcal F}_d(\eta_n,F_j,k),\, F\subset H_i \mbox{ for precisely $d-p$ indices } i\in\{1,\dots,k-j\}\big\}.
\end{eqnarray*}

We recall that $\Lambda_r(C)$, defined for $C\in \mathcal {PC}^d$ by (\ref{2.6}), is the normalized spherical $(r-1)$-volume of the $(r-1)$-skeleton of $C\cap\Sd$, that is,
$$ \Lambda_r(C)= \sum_{F\in{\mathcal F}_r(C)}V_r(F)    =\sum_{F\in{\mathcal F}_r(C)} \frac{\sigma_{r-1}(F\cap \Sd)}{\omega_r}.$$
We have
$$ \sum_{C\in {\mathcal F}_d(\eta_n,F_j,k)} \Lambda_r(C)=\sum_{C\in {\mathcal F}_d(\eta_n,F_j,k)}\;\sum_{F\in{\mathcal F}_r(C)} V_r(F) = \sum_{p=\max\{r,d-k+j\}}^{d} 2^{d-p} \sum_{F\in\mathcal{F}_{r,p}} V_r(F),$$
since each $F\in{\mathcal F}_{r,p}$ belongs to precisely $2^{d-p}$ cones $C\in{\mathcal F}_d(\eta_n,F_j,k)$.

Let $Q$ be the unique cone in ${\mathcal F}_d(\eta_n\langle 1,\dots, d-j \rangle)$ with $F_j\subset Q$, and define 
$$ {\mathcal C}_p := \big\{Q\cap H_{i_1}\cap\dots\cap  H_{i_{d-p}}: 1\le i_1<\dots<i_{d-p}\le k-j\big\}. $$ 
Thus, ${\mathcal C}_p$ is a set of $p$-dimensional cones, and ${\mathcal C}_d=\{Q\}$. Each $r$-face $F\in{\mathcal F}_{r,p}$ satisfies $F\subset G\in{\mathcal F}_r(D)$ for a unique $D\in{\mathcal C}_p$ and a unique $G\in{\mathcal F}_r(D)$. Conversely, for $D\in{\mathcal C}_p$ and $G\in{\mathcal F}_r(D)$, the $r$-face $G$ is the union of $r$-faces from ${\mathcal F}_{r,p}$, which pairwise have no relatively interior points in common. It follows that
$$ \sum_{F\in{\mathcal F}_{r,p}} V_r(F) =\sum_{D\in{\mathcal C}_p} \;\sum_{F\in{\mathcal F}_r(D)} V_r(F).$$
We conclude that
\begin{eqnarray}\label{7.2}
\sum_{C\in {\mathcal F}_d(\eta_n,F_j,k)} \Lambda_r(C) &=& \sum_{p=\max\{r,d-k+j\}}^{d} 2^{d-p} \sum_{D\in{\mathcal C}_p} \sum_{F\in\mathcal{F}_r(D)} V_r(F) \nonumber\\
&=& \sum_{p=\max\{r,d-k+j\}}^{d} 2^{d-p} \sum_{D\in{\mathcal C}_p}  \Lambda_r(D).
\end{eqnarray}

Relation (\ref{7.2}) was derived for any $F_j\in{\mathcal F}_j(\eta_n)$ with $F_j\subset L_j$. We sum over all such $j$-faces and note that $C\in {\mathcal F}_d(\eta_n\langle k-j+1,\dots,d-j\rangle)$ satisfies $F_j\subset C$ for some $j$-face $F_j\in{\mathcal F}_j(\eta_n)$ with $F_j\subset L_j$ if and only if $C\cap L_j\not=\{0\}$. (Recall that $\eta_n\langle i_1,\dots,i_{d-k}\rangle$ was defined early in Section \ref{sec6}.) Concerning the set $\mathcal{C}_p$ appearing on the right-hand side of \eqref{7.2}, we note that $Q\in {\mathcal F}_d(\eta_n\langle 1,\dots,d-j\rangle)$ satisfies $F_j\subset Q$ for some $j$-face $F_j\in{\mathcal F}_j(\eta_n)$ with $F_j\subset L_j$ if and only if $Q\cap L_j\not=\{0\}$. Therefore, we obtain the geometric identity
\begin{eqnarray}\label{7.3}
& & \sum_{C\in {\mathcal F}_d(\eta_n\langle k-j+1,\dots,d-j\rangle) \atop C\cap L_j\not=\{0\}} \Lambda_r(C)\\
& &  = \sum_{p=\max\{r,d-k+j\}}^{d} 2^{d-p} \sum_{1\le i_1<\dots< i_{d-p}\le k-j}\; \sum_{Q\in {\mathcal F}_d(\eta_n\langle 1,\dots,d-j\rangle)\atop Q\cap L_j\not=\{0\}} \Lambda_r(Q\cap H_{i_1}\cap\dots\cap H_{i_{d-p}}), \nonumber
\end{eqnarray}
which will be required in Section \ref{sec8}. 
(For $k=j$, the middle sum on the right-hand side has to be deleted, and the equation becomes a tautology.) This holds for $\eta_n=(H_1,\dots,H_n )\in G(d,d-1)^n_*$,  $j\in\{1,\dots,d-1\}$, with $ L_j:= H_1\cap\dots\cap H_{d-j}$,  $r=1,\dots,d$, $k\in \{j,\ldots,d\}$ and for $n>d-j$.

\section{A covariance matrix}\label{sec8}

We are now in a position to combine the preceding results, in order to finish the proof of Theorem \ref{T8.1}. The crucial task is to compute the expectation $\E \Lambda_r(D_n^{[k,j]})$ (formula (\ref{8.3})). To do this, we use the explicit representation (\ref{3.15c}) of the distribution of $D_n^{[k,j]}$ and employ the geometric decomposition result (\ref{7.3}) obtained in Section \ref{sec7}, together with properties of invariant measures.

We use (\ref{3.15c}), extended to expectations and then applied to the expectation of $\Lambda_r$, for given $r\in\{1,\dots,d\}$. However, it will be convenient to replace the index tuple $(1,\dots,d-k)$ by $(k-j+1,\dots,d-j)$, for given $j\in\{1,\dots,d-1\}$ and $k\in\{j,\dots,d\}$. As before we assume that $n> d-j$. Then we have (splitting the multiple integral appropriately)
\begin{eqnarray}\label{8.0}
&& {\mathbb E}\,\Lambda_r(D_n^{[k,j]})\\
&&= \frac{1}{2^{k-j}C(n-d+k,k,j)} \int_{G(d,d-1)^{n-d+j}} \int_{G(d,d-1)^{k-j}} \int_{G(d,d-1)^{d-k}}\nonumber\\
&& \hspace{4mm}\times\; \sum_{F\in{\mathcal F}_{d-k+j}(\eta_n\langle k-j+1,\dots,d-j\rangle)\atop F\cap H_{k-j+1}\cap\dots\cap H_{d-j}\not= \{0\}}\; \sum_{C\in{\mathcal F}_d (\eta_n\langle k-j+1,\dots,d-j\rangle)\atop C\supset F} \Lambda_r(C)\nonumber\\
&& \hspace{4mm}\times\; \nu_{d-1}^{d-k}(\D(H_{k-j+1},\dots,H_{d-j}))\,\nu_{d-1}^{k-j}(\D(H_1,\dots,H_{k-j}))\,\nu_{d-1}^{n-d+j}(\D(H_{d-j+1}, \dots,H_n)).\nonumber
\end{eqnarray}
(Recall that, for $k=d$, the condition $F\cap H_{k-j+1}\cap\dots\cap H_{d-j}\not= \{0\}$ is empty and can be deleted.) 
If $k>j$, we split the first sum above in the form
$$ \sum_{F\in{\mathcal F}_{d-k+j}(\eta_n\langle k-j+1,\dots,d-j\rangle)\atop F\cap H_{k-j+1}\cap\dots\cap H_{d-j}\not= \{0\}} = \sum_{1\le i_1<\dots<i_{k-j}\le n \atop i_1,\dots,i_{k-j}\notin\{k-j+1,\dots, d-j\}}
\sum_{F\in{\mathcal F}_{d-k+j}(\eta_n\langle k-j+1,\dots,d-j\rangle)\atop F\cap H_{k-j+1}\cap\dots\cap H_{d-j}\not= \{0\},\; F\subset H_{i_1}\cap\dots\cap H_{i_{k-j}}}.
$$
Then, after interchanging in (\ref{8.0}) the first summation on the right side and integration, the outer sum has $\binom{n-d+k}{k-j}$ equal terms, hence we obtain (again regrouping the integrals)
\begin{eqnarray*}
 {\mathbb E}\,\Lambda_r(D_n^{[k,j]})
& =& \frac{\binom{n-d+k}{k-j}}{2^{k-j}C(n-d+k,k,j)} \int_{G(d,d-1)^{k-j}}  \int_{G(d,d-1)^{n-k+j}} \allowdisplaybreaks\\
&& \times\; \sum_{F\in{\mathcal F}_{d-k+j}(\eta_n\langle k-j+1,\dots,d-j\rangle)\atop F\subset H_{1}\cap\dots\cap H_{k-j},\;F\cap H_{k-j+1}\cap\dots\cap H_{d-j}\not= \{0\}}\; \sum_{C\in{\mathcal F}_d (\eta_n\langle k-j+1,\dots,d-j\rangle)\atop C\supset F} \Lambda_r(C)\\
&&\times\; \nu_{d-1}^{n-k+j}(\D(H_{k-j+1}, \dots,H_n))\,\nu_{d-1}^{k-j}(\D(H_1,\dots,H_{k-j})).
\end{eqnarray*}
(If $k=j$, the condition $F\subset H_{1}\cap\dots\cap H_{k-j}$ is empty and can be deleted.)
For fixed subspaces $H_1,\dots,H_{k-j}$, we consider the inner integral
\begin{eqnarray*}
I &:=&  \int_{G(d,d-1)^{n-k+j}} \sum_{F\in{\mathcal F}_{d-k+j}(\eta_n\langle k-j+1,\dots,d-j\rangle)\atop F\subset H_1\cap\dots\cap H_{k-j},\;F\cap H_{k-j+1}\cap\dots\cap H_{d-j}\not= \{0\}}\; \sum_{C\in{\mathcal F}_d (\eta_n\langle k-j+1,\dots,d-j\rangle)\atop C\supset F} \Lambda_r(C)\\
&& \times\; \nu_{d-1}^{n-k+j}(\D(H_{k-j+1},\dots,H_n)).
\end{eqnarray*}
A cone $C\in{\mathcal F}_d(\eta_n\langle k-j+1,\dots,d-j\rangle)$ has a face 
$ F \in {\mathcal F}_{d-k+j} (\eta_n\langle k-j+1,\dots,d-j\rangle)$ satisfying 
$$ F\subset H_1\cap\dots\cap H_{k-j} \quad \text{and} \quad F \cap H_{k-j+1} \cap\dots\cap H_{d-j} \not= \{0\}$$
if and only if
$$ C\cap H_1\cap\dots\cap H_{d-j}\not=\{0\},$$
and it can have at most one such face. Using this and (\ref{7.3}), we obtain
\begin{eqnarray*}
I &=& \int_{G(d,d-1)^{n-k+j}} \sum_{C\in{\mathcal F}_d(\eta_n\langle k-j+1,\dots,d-j\rangle) \atop C\cap H_1\cap\dots\cap H_{d-j} \not= \{0\}} \Lambda_r(C) \, \nu_{d-1}^{n-k+j}(\D(H_{k-j+1},\dots,H_n))\\
&=& \sum_{p=\max\{r,d-k+j\}}^{d} 2^{d-p} \sum_{1\le i_1<\dots< i_{d-p}\le k-j} \int_{G(d,d-1)^{n-k+j}}\\ 
&& \times\; \sum_{Q\in {\mathcal F}_d(H_{d-j+1},\dots,H_n) \atop Q\cap H_1\cap\dots\cap H_{d-j}\not=\{0\}}\Lambda_r(Q\cap H_{i_1} \cap \dots\cap H_{i_{d-p}})\, \nu_{d-1}^{n-k+j}(\D(H_{k-j+1},\dots,H_n)).
\end{eqnarray*}
We conclude that
\begin{eqnarray*}
&& {\mathbb E}\,\Lambda_r(D_n^{[k,j]})\\
&&= \frac{\binom{n-d+k}{k-j}}{2^{k-j}C(n-d+k,k,j)}  \sum_{p=\max\{r,d-k+j\}}^{d} 2^{d-p} \sum_{1\le i_1<\dots< i_{d-p}\le k-j} \\ 
&& \hspace{4mm}\times\;\int_{G(d,d-1)^n} \sum_{Q\in {\mathcal F}_d(H_{d-j+1},\dots,H_n) \atop Q\cap H_1\cap\dots\cap H_{d-j}\not=\{0\}}\Lambda_r(Q\cap H_{i_1} \cap \dots\cap H_{i_{d-p}})\,\nu_{d-1}^n(\D(H_1,\dots,H_n))\nonumber\allowdisplaybreaks\\
&&= \frac{\binom{n-d+k}{k-j}}{2^{k-j}C(n-d+k,k,j)}  \sum_{p=\max\{r,d-k+j\}}^{d} 2^{d-p} \binom{k-j}{d-p} \int_{G(d,d-1)^{d-p}}\\ 
&& \hspace{4mm}\times\;\int_{G(d,d-1)^{n-d+p}} \sum_{Q\in {\mathcal F}_d(H_{d-j+1},\dots,H_n) \atop Q \cap H_1 \cap \dots \cap H_{d-j}\not=\{0\}} \Lambda_r(Q\cap H_1 \cap \dots \cap H_{d-p})\\ 
&& \hspace{4mm}\times\;\nu_{d-1}^{n-d+p}(\D(H_{d-p+1},\dots, H_n))\,\nu_{d-1}^{d-p}(\D(H_1,\dots,H_{d-p})).
\end{eqnarray*}

To evaluate the inner integral above, we fix $H_1,\dots,H_{d-p}$ in general position and write $H_1\cap\dots\cap H_{d-p} =: L_p$. The image measure of $\nu_{d-1}$ under the  ($\nu_{d-1}$ almost everywhere well defined) map $H\mapsto H\cap L_p$ from $G(d,d-1)$ to the Grassmannian $G(L_p,p-1)$ of $(p-1)$-dimensional subspaces of $L_p$ is the invariant probability measure $\mu_{p-1}$ on $G(L_p,p-1)$. Therefore, the inner integral can be written as
\begin{eqnarray*}
&& \int_{G(d,d-1)^{n-d+p}} \sum_{Q\in {\mathcal F}_d(H_{d-j+1},\dots,H_n) \atop Q \cap H_1 \cap \dots \cap H_{d-j}\not=\{0\}} \Lambda_r(Q\cap L_p)\, \nu_{d-1}^{n-d+p}(\D(H_{d-p+1},\dots, H_n)) \\
&& =\int_{G(L_p,p-1)^{n-d+p}} \sum_{C\in {\mathcal F}_p(h_{d-j+1},\dots,h_n) \atop C \cap h_{d-p+1} \cap \dots \cap h_{d-j}\not=\{0\}} \Lambda_r(C)\, \mu_{p-1}^{n-d+p}(\D(h_{d-p+1},\dots, h_n)).
\end{eqnarray*}
Note that $p\ge j$. If $p=j$, then the second condition under the last sum is empty and can be deleted. 
Here ${\mathcal F}_p(h_{d-j+1},\dots,h_n)$ denotes the set of Schl\"afli cones in $L_p$ that are generated by the $(p-1)$-planes $h_{d-j+1},\dots,h_n$ in $L_p$.
Identifying $L_p$ with ${\mathbb R}^p$, we can apply (\ref{6.54}) in $L_p$. For this, we replace $d$ by $p$, the number $n$ by $n-d+p$, and raise the indices of the integration variables in (\ref{6.54}) by $d-p$. Then (\ref{6.54}), with $g=\Lambda_r$, reads
\begin{eqnarray*}
&& \int_{G(L_p,p-1)^{n-d+p}} \sum_{C\in {\mathcal F}_p(h_{d-j+1},\dots,h_n)  \atop C\cap h_{d-p+1}\cap\dots\cap h_{d-j} \not=\{0\}} \Lambda_r(C)\, \mu_{p-1}^{n-d+p}(\D(h_{d-p+1},\dots,h_n))\\
&& = 2 C(n-d+j,p)\,{\mathbb E} (\Lambda_rU_{p-j})(S^{(p)}_{n-d+j}),
\end{eqnarray*}
where $S^{(p)}_m$ denotes the $(\mu_{p-1},m)$- Schl\"afli cone in $L_p$.  We conclude that
\begin{eqnarray}\label{8.3}
{\mathbb E}\,\Lambda_r(D_n^{[k,j]})&=& \frac{2\binom{n-d+k}{k-j}}{2^{k-j}C(n-d+k,k,j)}  \sum_{p=\max\{r,d-k+j\}}^{d} 2^{d-p} \binom{k-j}{d-p}\nonumber\\
&& \times\; C(n-d+j,p)\, {\mathbb E}(\Lambda_rU_{p-j})(S^{(p)}_{n-d+j}).
\end{eqnarray}
Comparing (\ref{8.3}) and (\ref{3.16}), and recalling that $n>d-j$, we arrive at
\begin{eqnarray*}
& & {\mathbb E}(\Lambda_r Y_{d-k+j,d-k})(S_{n-d+k}) \\
&& = \frac{\binom{n-d+k}{k-j}}{C(n-d+k,d)} \sum_{p=\max\{r,d-k+j\}}^d 2^{d-p}\binom{k-j}{d-p} C(n-d+j,p) \,{\mathbb E}(\Lambda_rU_{p-j})(S^{(p)}_{n-d+j}).
\end{eqnarray*}
Here we substitute $d-k+j=s$ and $d-k=t$. Then we replace $n$ by $n+t$ and assume that $n> d-s$. The result is
\begin{eqnarray*}
& & {\mathbb E}(Y_{s,t}\Lambda_r)(S_n) \\
&& = \frac{\binom{n}{d-s}}{C(n,d)} \sum_{p=\max\{r,s\}}^d 2^{d-p}\binom{d-s}{d-p} C(n-d+s,p) \,{\mathbb E} (U_{p-s+t}\Lambda_r)(S^{(p)}_{n-d+s}).
\end{eqnarray*}
This is the conical (or spherical) counterpart to \cite[ Thm.~11.7.1]{Mil61}. (The result is also true for $n<d-s$, since then both sides of the equation are zero.)  

We specialize the latter to $t=s-1$. We have $Y_{s,s-1}=\Lambda_s$. Further, $U_{p-s+t}=U_{p-1}= V_p$ in a space of dimension $p$. The value of ${\mathbb E}(V_p\Lambda_r)(S^{(p)}_{n-d+s})$ is seen from (\ref{5.6}). In this way, we obtain the following result.

\begin{theorem}\label{T8.1}
The face contents of the $(\nu_{d-1},n)$-Schl\"afli cone $S_n$ satisfy
\begin{align}\label{8.43} 
&{\mathbb E}(\Lambda_s\Lambda_r)(S_n) \\
&= \frac{\binom{n}{d-s}}{C(n,d)} \sum_{p=\max\{r,s\}}^d 2^{d-p}\binom{d-s}{d-p}
\binom{n-d+s}{p-r}\theta(n-d-p+r+s,p)\nonumber
\end{align}
for $r,s=1,\dots,d$, where $\theta$ is defined by $(\ref{varpi})$.
\end{theorem}

An alternative formulation of (\ref{8.43}), which exhibits the symmetry in $r$ and $s$, is given by
\begin{align}\label{8.43a} 
&{\mathbb E}(\Lambda_s\Lambda_r)(S_n) \\
&= \frac{1}{C(n,d)} \sum_{p\in{\mathbb N}} 2^{d-p}\binom{n}{d-p}\binom{n-d+p}{p-s,\,p-r,\,n-d-p+r+s}\theta(n-d-p+r+s,p)\nonumber.
\end{align}

Theorem \ref{T8.1} is the conical counterpart to \cite[ Corollary to Thm.~11.7.1]{Mil61}. It holds for all $n\in\mathbb{N}$. In fact, if $n<d-r$ (or $n<d-s$), then both sides of \eqref{8.43} are zero. For $n=d-r$ (or $n=d-s$) equation \eqref{8.43}   is equivalent to \eqref{3.7a}. Also note that \eqref{5.6} is obtained as the special case $s=d-k$ and $r=d$ of \eqref{8.43}. 

Since the expectations ${\mathbb E} \Lambda_r(S_n)$ are known by (\ref{3.7a}), Theorem 8.1 allows us to write down the complete covariance matrix for the random vector $(\Lambda_1(S_n),\dots,\Lambda_d(S_n))$.

For the Cover--Efron cone $C_n$, there is only one second moment that we can obtain from Theorem 8.1 by dualization, namely ${\mathbb E}f_{d-1}^2(C_n)= {\mathbb E} f_1^2(S_n)= 4{\mathbb E} \Lambda_1^2(S_n)$ for $n\ge d$.

\bigskip \vspace*{1cm}

\noindent Authors' addresses:\\[2mm]
\noindent Daniel Hug\\
Karlsruhe Institute of Technology \\
Department of Mathematics \\
D-76128 Karlsruhe, Germany\\
e-mail: daniel.hug@kit.edu\\[2mm]
\noindent Rolf Schneider\\
Albert-Ludwigs-Universit\"at\\
Mathematisches Institut\\
D-79104 Freiburg i. Br., Germany\\
e-mail: rolf.schneider@math.uni-freiburg.de

\end{document}